\documentclass[12pt]{amsart}

\usepackage{geometry}
\usepackage[utf8]{inputenc}
\usepackage[T1]{fontenc}
\usepackage{textcomp}

\usepackage{enumerate}
\usepackage[utf8]{inputenc}

\usepackage{graphicx,amssymb,amstext,amsmath,amsthm,amscd,amsfonts,hyperref, mathrsfs}
\usepackage{times}
\usepackage{wrapfig}
\usepackage{enumitem}
\usepackage{float}
\usepackage[usenames, dvipsnames]{color}
\usepackage{xcolor}
\usepackage{listings}
\usepackage{tikz,tkz-tab}
\usepackage{tikz-cd}
\usepackage{emptypage}
\usepackage{fancyhdr}
\usepackage{xcolor}
\definecolor{UBCblue}{rgb}{0.04706, 0.13725, 0.26667}
\definecolor{bistre}{rgb}{0.24, 0.17, 0.12}
\definecolor{bulgarianrose}{rgb}{0.28, 0.02, 0.03}

\fancypagestyle{appendix}{%
	\fancyhead{}
	
	\fancyhf{}
	\fancyhead[LO]{\itshape \nouppercase{\rightmark}}
	\fancyhead[RE]{\sc \nouppercase{\leftmark}}
	\fancyfoot[CE,CO]{\nouppercase{\thepage}}
}
\usepackage{calc}
\usepackage{tikz}
\usetikzlibrary{decorations.markings}
\tikzstyle{vertex}=[circle, draw, inner sep=0pt, minimum size=6pt]

\newcommand{\K} {{\mathcal{K}}}

\newcommand{\X}{\mathcal{X}}

\newcommand{\I}{{I}}

\newcommand{\norm}[1]{\left\lVert #1\right\rVert}
\newcommand{\inp}[2]{\left\langle {#1} ~ ,\,{#2} \right\rangle}

\numberwithin{equation}{section}
\newtheorem{theorem}{\bf Theorem}[section]

\newtheorem{proposition}{\bf Proposition}[section]
\newtheorem{lemma}{\bf Lemma}[section]

\newcommand\Item[1][]{%
	\ifx\relax#1\relax  \item \else \item[#1] \fi
	\abovedisplayskip=0pt\abovedisplayshortskip=0pt~\vspace*{-\baselineskip}}

\definecolor{black}{rgb}{0.0, 0.0, 0.0}

\usepackage{titlesec}
\titleformat{\section}{\color{bulgarianrose}\normalfont\Large\bfseries \filcenter}{}{0em}{\thesection\hspace*{4.3mm}}
\titleformat{\subsection}{\color{UBCblue}\normalfont\large\bfseries}{}{0em}{\thesubsection\hspace*{3.3mm}}
\titleformat*{\subsubsection}{\bfseries}

\theoremstyle{plain}
\numberwithin{equation}{section}
\newtheorem{Corollary}{\bf Corollary}[section]
\begin{document}
	
	%-------------------------------------------------------------------------
	% editorial commands: to be inserted by the editorial office
	%
	%\firstpage{1} \volume{228} \Copyrightyear{2004} \DOI{003-0001}
	%\seriesextraline{This is the Concrete Title of this Book\br H.E. R and S.T.C. W, Eds.}
	%\firstpage{6}
	%\issuenumber{6}
	%\Volumeandyear{1 (2022)}
	%\Copyrightyear{2022}
	%\submitted{September 04, 2022}
	%\received{September 04, 2022}
	%\revised{December 1, 2022}
	%\accepted{Dedember 02, 2022}
	
	\title[{\small Asymptotic Error Analysis - Discrete Iterated Galerkin Solution}]{Asymptotic error analysis for the discrete iterated Galerkin solution of Urysohn integral equations with Green's kernels}
	%----------Author 1
	\author[G. Rakshit]{Gobinda Rakshit}
	
	\address{%
		Department of Mathematical Sciences\\
		Rajiv Gandhi Institute of Petroleum Technology\\
		Jais Campus, Uttar Pradesh 229304\\
		India.\\
		ORCID iD : 0000-0002-5813-4656}
	\email{g.rakshit@rgipt.ac.in}
	%
	%\thanks{This file has been typeset with the option \texttt{draft} to illustrate that feature and its purpose.}
	
	\date{\today}
	%\dedicatory{Last Revised:\\ \today}
	%%% ----------------------------------------------------------------------
	\begin{abstract}
		Consider a Urysohn integral equation $x - \mathcal{K} (x) = f$, where $f$ and the integral operator $\mathcal{K}$ with kernel of the type of Green's function are given. In the computation of approximate solutions of the given integral equation by Galerkin method, all the integrals are needed to be evaluated by some numerical integration formula. This gives rise to the discrete version of the Galerkin method. For $r \geq 1$, a space of piecewise polynomials of degree $\leq r-1$ with respect to a uniform partition is chosen to be the approximating space. For the appropriate choice of a numerical integration formula, an asymptotic series expansion of the discrete iterated Galerkin solution is obtained at the above partition points. Richardson extrapolation is used to improve the order of convergence. Using this method we can restore the rate of convergence when the error is measured in the continuous case. A numerical example is given to illustrate this theory. 
	\end{abstract}
	\label{page:firstblob}
	%%% ----------------------------------------------------------------------
	\maketitle
	%%% ----------------------------------------------------------------------
	%\tableofcontents
	
	\section{Introduction}\label{sec1}
	
	Let $\X = L^{\infty}[0, 1]$. Consider the problem of solving Urysohn integral equation  
	\begin{align}\label{eq:main}
		x(s)- \int_{0}^{1} \kappa(s, t, x(t)) \: dt = f(s), \hspace{0.7cm} ~ s\in [0,1],
	\end{align}
	where $f \in \X$ and $\kappa \in C\left([0, 1] \times [0, 1] \times \mathbb{R}\right)$ are given. Let the Urysohn integral operator $\mathcal{K} :  L^{\infty}[0,1] \rightarrow C[0,1]$ be defined by
	\begin{equation}\label{eq:1.1}
		\mathcal{K}({x})(s)= \int_{0}^{1} \kappa(s, t, x(t)) ~ dt, \hspace{0.7cm} x \in \X, ~ s\in [0,1].
	\end{equation}
	Since the kernel $\kappa$ is continuous, $\mathcal{K}$ is compact operator on $\X$. Denoting the equation \eqref{eq:main} by
	\begin{equation}\label{eqn:main}
		x - \mathcal{K} x = f.
	\end{equation}
	We assume that the above equation has a solution, say $\varphi$. We also assume that $\mathcal{K}$ is twice Frech\'et differentiable and $ 1$ is not an eigenvalue of the compact linear operator $\mathcal{K}'{(\varphi)}.$ This gives us that $\varphi$ is an isolated solution of \eqref{eqn:main}. See \cite{Kra}, \cite{KraZ}. We are looking for Galerkin approximations of $\varphi$.
	
	For $ r \geq 1$, consider the approximating space $ \mathcal{X}_n$ as a 
	space of piecewise polynomials of degree $ \leq r-1 $ with respect to a 
	uniform partition, say $\Delta^{(n)}$, of $ [0, 1]$ with $n$ subintervals each of length $ {h = \frac {1} {n} }.$ Let $\pi_n$ be the restriction to $L^\infty [0, 1]$ of the orthogonal projection from $L^2 [0, 1]$ to $\mathcal{X}_n.$ Then the Galerkin solution $\varphi_n^G$ satisfies the following integral equation
	\begin{equation*}
		\varphi_n^G - \pi_n \mathcal {K} (\varphi_n^G)=\pi_n f. 
	\end{equation*}
	Galerkin method for Urysohn integral equation has been studied extensively in research literature. See \cite{Atk-Pot}, \cite{Kra}, \cite{KraV}, \cite{KraZ}. The iterated Galerkin solution is defined by
	$$\varphi_n^S = \mathcal{K} (\varphi_n^G) + f. $$
	In \cite{Atk-Pot}, the following orders of convergence are also obtained. 
	\begin{equation}\nonumber
		\| \varphi_n^G - \varphi \|_\infty = O \left( h \right), \;\;\; 
		\| \varphi_n^S  - \varphi \|_\infty = O \left( h^{2} \right), \quad \text{ if } r = 1, \end{equation} 
	and
	\begin{equation}\nonumber
		\| \varphi_n^G - \varphi \|_\infty = O \left( h^{  r} \right), \;\;\; 
		\| \varphi_n^S  - \varphi \|_\infty = O \left(h^{ r+2 }\right), \quad \text{if } r \geq 2.
	\end{equation}
	It is also shown that the order of convergence of $\varphi_n^S$ at the points of partition $\Delta^{(n)}$, is $h^{2r}$.
	
	If an asymptotic expansion for the error exists, one can apply a well-known techniques to obtain more accurate approximations. Richardson extrapolation one such method for application. In \cite{rakshit2021richardson}, an asymptotic expansion for the iterated Galerkin solution of Urysohn integral equation with Green's function type of kernel, is obtained at the above mentioned partition points. Then, by \cite{ford2000asymptotic} and using Richardson extrapolation, an approximate solution with order of convergence $h^{2r+2}$ can be obtained.

	In the computation of  of above approximations, various integrals are involved. There is an integral in the definition of the Urysohn integral operator $\mathcal{K}$. In the definition of the orthogonal projection $\pi_n$, the standard inner product on $L^2[0, 1]$ comes into picture. In practice, it is necessary to replace all these integrals by a numerical quadrature formula. This gives rise to the discrete versions of the projection methods. The discrete versions of the Galerkin methods for Urysohn integral with Green's kernel, are investigated in \cite{Atk-Pot3}, \cite{Atk-Pot2}. Whereas, in \cite{RPK-GR3}, a different version of discrete projection method is discussed. 
	
	In this article, we consider the Urysohn integral equation with Green's kernel, and discrete Galerkin method is applied for approximations. Then, an asymptotic expansion for the discrete iterated Galerkin solution is obtained.
	
	We choose a fine partition of $[0, 1]$ with $m$ subintervals each of length $\tilde{h} = \frac{1}{m}$ and define a composite  numerical quadrature formula. Replacing the integrals in the definition of $\mathcal{K}$ and $\pi_n$, we define the Nystr\"om operator $\mathcal{K}_m$ and the discrete orthogonal projection $P_n$. Then the discrete Galerkin and the discrete iterated Galerkin equations are given by $$z_n^G - P_n \mathcal{K}_m (z_n^G) = P_n f  ~~~ \text{ and }~~~ z_n^S - \mathcal{K}_m(P_n z_n^S) = f$$ respectively. 
	If $\varphi \in C^{r+2}[0, 1]$, then from \cite{Atk-Pot3} and \cite{RPK-GR3}, we have
	\begin{equation}\label{eq:discrete_Gal}
		\norm{z_n^G - \varphi}_\infty = O\left( \max\left\{ h^{r}, \tilde{h}^2 \right\}  \right)		
	\end{equation}
	\begin{equation}\label{eq:discrete_it_Gal}
		\norm{z_n^S - \varphi}_\infty = \left\{ {\begin{array}{ll}
				O \left(\max\left\{ h^{2}, \tilde{h}^2 \right\}\right), ~\;\;\;  r = 1,   \vspace*{1.5mm}\\
				O \left(\max\left\{ h^{r+2}, \tilde{h}^2 \right\}\right), ~  r \geq 2.
		\end{array}}\right.
	\end{equation}
	
	In this article, first we find an asymptotic error expansion due to the discrete orthogonal projection. Then using this, the following asymptotic expansion is obtained:
	\begin{equation}\label{eq:final}
		z_n^S(t_i) = \varphi(t_i) + \gamma(t_i) h^{2r} +  O \left(\max\left\{ h^{2r+2}, \tilde{h}^2 \right\}\right),
	\end{equation}
	where the function $\gamma$ is independent of $h$. If we choose $m$ such that $\tilde{h} \leq h^{2r+2}$, then using the Richardson extrapolation, an approximation of $\varphi$ of the order of $h^{2r+2}$ could be obtained. See \cite{ford2000asymptotic}. 
	
	This article is organized as follows. Definitions, notations and some preliminary results are given in section 2. In Section 3, a quadrature rule is defined, and using it the discrete orthogonal projection and the Nystr\"om approximations of the integral operators are defined. Section 4 contains the asymptotic error analysis for the approximations. Numerical example is given in Section 5.
	
	\section{Preliminaries}
	
	For an integer $\alpha \geq 0$, let $ C^{\alpha}[0, 1]$ denotes the space of all real valued $\alpha$-times continuously differentiable functions on $[0, 1]$ with the norm 
	$$\norm{x}_{\alpha, \infty} = \max_{0 \leq j \leq \alpha} \norm{x^{(j)}}_\infty,$$ where $x^{(j)}$ is the $j^{\text{th}}$ derivative of the function $x$, and {\small $\displaystyle{\norm{x^{(j)}}_\infty = \sup_{0 \leq t \leq 1} \lvert x^{(j)}(t) \rvert}$}. 
	Define
	{\small \begin{equation*}
			\norm{\kappa}_{\alpha, \infty} = \max_{0 \leq i+j+k \leq \alpha} \norm{D^{(i, j, k)}\kappa(s, t, u)}_\infty,
		\end{equation*}
		where
		$$D^{(i, j, k)}\kappa(s, t, u) = \frac{\partial^{i+j+k}  \kappa}{\partial s^i \partial t^j \partial u^k}(s, t, u).$$}
	
	\subsection*{Green's function type kernel}\label{subsection:2.1}

	Let $r \geq 1$   be an integer and assume that the kernel $\kappa$ has the following properties.
	
	\begin{enumerate}
		\item For $i = 1, 2, 3, 4$, the functions $\kappa, \frac { \partial^i \kappa} {\partial u^i} \in C ( \Omega),$
		where $C ( \Omega)$ denotes the space of all real valued continuous function on $\Omega = [0, 1] \times [0, 1] \times \mathbb{R}$.
		
		\item Let
		$ \Omega_1 = \{ (s, t, u): 0 \leq t \leq s \leq 1, \; u \in \mathbb{R} \}$ and
		$\Omega_2 = \{ (s, t, u): 0 \leq s \leq t \leq 1, \; u \in \mathbb{R} \}.$
		There are two functions $\kappa_j \in C^{r} ( \Omega_j ), j = 1, 2, $ such that
		\begin{equation*}
			\kappa (s,t, u) = \left\{ {\begin{array}{ll}
					\kappa_1 (s, t, u), \;\;\; (s, t, u) \in \Omega_1,   \\
					\kappa_2 (s, t, u), \;\;\; (s, t, u) \in \Omega_2.
			\end{array}}\right.
		\end{equation*}
		\item Denote
		$ \ell (s, t, u) = \frac {\partial \kappa  } { \partial u}( s, t, u) $ and $ \lambda (s, t, u) = \frac {\partial^2 \kappa } { \partial u^2}( s, t, u), \:(s, t, u) \in \Omega.$ The partial derivatives of $\ell (s, t, u)$ and $\lambda (s, t, u)$ with respect to $s$ and $t$ have jump discontinuities on $s = t$.
		
		\item There are functions $\ell_j, \lambda_j \in C^{r}  ( \Omega_j ), j = 1, 2, $ with
		{\footnotesize	\begin{equation*}
				\ell (s,t, u) = \left\{ {\begin{array}{ll}
						\ell_1 (s, t, u), \;\;\; (s, t, u) \in \Omega_1,   \\
						\ell_2 (s, t, u), \;\;\; (s, t, u) \in \Omega_2,
				\end{array}}\right. 
				\hspace*{0.6mm}	\lambda (s,t, u) = \left\{ {\begin{array}{ll}
						\lambda_1 (s, t, u), \;\;\; (s, t, u) \in \Omega_1,   \\
						\lambda_2 (s, t, u), \;\;\; (s, t, u) \in \Omega_2.
				\end{array}}\right.
		\end{equation*}}
		\end {enumerate}

		Under the above assumptions, the operator $\mathcal {K}$ is four times Fr\'echet differentiable, and its Fr\'echet derivatives at $x \in \mathcal{X}$ are given by 
		$$ \mathcal {K}'(x) v_1 (s) = \int_0^1 \frac {\partial \kappa } {\partial u} \left(s,t, x(t)\right) \:
		v_1(t) \: dt, $$
		\begin{equation*}
			\mathcal {K}^{(i)}(x) (v_1,\ldots, v_i) (s) = \int_0^1 \frac {\partial^i \kappa } {\partial u^i} \left(s,t,x(t)\right) \:
			v_1(t) \cdots v_i(t) \: dt, \quad i = 2, 3, 4,
		\end{equation*}
		where 
		\begin{align*}
			\frac {\partial^i \kappa } {\partial u^i} \left(s,t,x(t)\right) = \frac {\partial^i \kappa } {\partial u^i} \left(s,t,u\right)\rvert_{u = x(t)}, \quad i = 1, 2, 3, 4
		\end{align*} and $v_1, v_2, v_3, v_4 \in \mathcal{X}$. Note that $\mathcal{K}' (x) : \mathcal{X} \rightarrow \mathcal{X}$ is linear and $ \mathcal{K}^{(i)}(x) : \mathcal{X}^i \rightarrow \mathcal{X} $ are multi-linear operators, where $\mathcal{X}^i$ is the cartesian product of $i$ copies of $\mathcal{X}$. See \cite{Rall}. The norms of these operators are defined by $	\norm{\mathcal {K}^{(i)}(x) } = \sup \left\{ \norm{\mathcal {K}^{(i)}(x) (v_1, \ldots, v_i)}_\infty :   \norm{v_j}_\infty \leq 1, j = 1, \ldots, i \right\} $ for $i = 1, 2, 3, 4$. It follows that
		\begin{eqnarray}\nonumber
			\norm{\mathcal {K}^{(i)}(x) } &\leq& \sup_{0 \leq s, t \leq 1} \left\lvert \frac {\partial^i \kappa } {\partial u^i} \left(s, t, x(t)\right) \right\rvert, \qquad i = 1, 2, 3, 4.
		\end{eqnarray}
		Note that, if $ f \in C^{\alpha} [0, 1]$ for any positive integer $\alpha$, then $\varphi \in C^{\alpha} [0, 1]$. See \cite[Corollary 3.2]{Atk-Pot}, \cite[Corollary 4.2]{Atk-Pot3}.
		
		%	Let $ f \in C^{\alpha} [0, 1]$, then by \cite[Corollary 3.2]{Atk-Pot}, $\varphi \in C^{\alpha} [0, 1].$
		\section{Discretization of Integrals by numerical quadrature Rule}	
		In this section, first we consider a numerical integration formula. We replace the integral in the standard inner product of $L^2[0, 1]$ (i.e. $ \inp{x}{y} = \int_{0}^{1} x(t) y(t) \:dt $) by the quadrature rule and define a discrete inner product. Subsequently, the corresponding discrete orthogonal projection is defined. After that, an asymptotic error expansion for the discrete orthogonal projection is obtained. Next we define the Nystr\"om approximations of the integral operator $\mathcal{K}$ and its Fr\'echet derivatives.

		Consider a basic numerical integration formula by
		\begin{equation}\label{eq:Num_int}
			\int_0^1 x (t) d t \approx \sum_{q=1}^{\rho} w_q \: x(\mu_q),
		\end{equation}
		which is exact at least for polynomials of degree $\leq 3r. $ 
		If $r = 0,$ then it is assumed that the quadrature rule is exact atleast for linear polynomials. It follows that $\sum_{q=1}^{\rho} w_q = 1$.

		Let $n \in \mathbb{N}$ and consider the following uniform partition of $[0, 1]:$
		\begin{equation}\label{eq:partition1}
			\Delta^{(n)} ~ : \qquad 0  <  \frac{1} {n}  < \cdots <   \frac{n-1} {n}   <  1.
		\end{equation}
		Define
		$ 	t_j = \frac {j} {n}, \; \Delta_j = [t_{j-1}, t_j] $ and $h = t_{j} - t_{j-1} = \frac {1} {n}, \; j = 1, \ldots, n.$ Define the subspace $C^{\alpha}_{\Delta^{(n)}}[0, 1] = \left\{ x \in \mathcal{X} : x \in C^{\alpha}[t_{j-1}, t_j], ~ j = 1, 2, 3, \dots, n \right\}.$
		For $r \geq 1,$ the approximating space $$\mathcal{X}_n = \left\{ x \in \X : x\vert_{\Delta_j} \text{ is a polynomial of degree } \leq r-1 \right\}.$$ Let $p$ be a positive integer and $m = p n$. Consider the following uniform partition of $[0, 1]:$
		\begin{equation}\label{eq:fine_part}
			\Delta^{(m)} ~ : \qquad	0  <  \frac{1} {m}  < \cdots <   \frac{m-1} {m}   <  1.
		\end{equation}
		Let 
		$\displaystyle { \tilde{h} = \frac {1} {m} \;\;\; \mbox {and} \;\;\; s_i = \frac {i} {m},  \;\;\; i = 0, \ldots, m.}$ \\
		Note : As our goal to find the equation \eqref{eq:final}, where the higher order term is $\max\left\{ {h^{2r+2}, \tilde{h}^2} \right\}$, we choose the partition $\Delta^{(m)}$ such that $\tilde{h}^2 \leq h^r$.
		
		A composite integration rule with respect to the partition  (\ref{eq:fine_part}) 
		is then defined as
		\begin{eqnarray*}
			\int_0^1 x(t) \: d t &=& \sum_{i =1}^m \int_{s_{i -1}}^{s_i} x (t) \: d t 
			\approx  \tilde h \sum_{i=1}^m  \sum_{q =1}^{\rho} w_q \: x (s_{i -1} + \mu_q \tilde{h} ).
		\end{eqnarray*}
		%Then $$[(j-1)h, jh] = \bigcup_{\nu=1}^{p} [(j-1)h + (\nu-1)\tilde{h}, (j-1)h + \nu \tilde{h}].$$
		Thus,
		\begin{eqnarray}\nonumber
			\int_{t_{j -1}}^{t_j} x (t) \: d t  = 	\int_{(j-1)h}^{jh} x(t) \:  d t   =   \sum_{\nu =1}^p 	\int_{(j-1)h + (\nu-1)\tilde{h}}^{(j-1)h + \nu \tilde{h}} x (t) \: d t.
		\end{eqnarray}
		Since $h = p \tilde{h}$,
		\begin{eqnarray*}
			\int_{t_{j -1}}^{t_j} x (t) d t  & = & \sum_{\nu =1}^p 	\int_{\frac{(j-1)p + \nu-1}{p}h}^{\frac{(j-1)p + \nu}{p}h} x(t) \: dt.
		\end{eqnarray*}
		Substituting $	t  \: = \:   \frac{(j-1)p + \nu-1}{p}h + \tilde{h} \sigma \: = \: \frac{(j-1)p + \nu-1 + \sigma}{p} h $ in the above equation, we obtain
		\begin{eqnarray*}
			\int_{t_{j -1}}^{t_j}  x(t) \: d t  & = &  \frac{h}{p} \: \sum_{\nu =1}^p 	\int_{0}^{1}  x {\left( \frac{(j-1)p + \nu-1 + \sigma}{p} h \right)} \: d \sigma \nonumber\\
			& = &  \frac{h}{p} \: \sum_{\nu =1}^p 	\int_{0}^{1} x{ \left( t_{j-1} +\frac{\nu - 1 + \sigma}{p} h \right)} \: d \sigma.
		\end{eqnarray*}
		Note that $\frac{\nu - 1 + \sigma}{p} \in [0, 1].$
		Now using the numerical quadrature formula \eqref{eq:Num_int}, we obtain
		\begin{eqnarray*}
			\int_{t_{j -1}}^{t_j}  x(t) \: d t  & \approx &  \frac{h}{p} \: \sum_{\nu =1}^p 	\sum_{q =1}^\rho w_q \:  x{\left( t_{j-1} + \frac{\nu-1 + \mu_q}{p} h \right)}.
		\end{eqnarray*}
		Let $$ \mu_{q \nu} = \frac{\nu-1 + \mu_q}{p}, \quad q = 1, 2, \ldots, \rho; ~ \nu = 1, 2, \ldots, p. $$
		Then,
		\begin{eqnarray}\label{eq:Comp_Num_int}
			\int_{t_{j -1}}^{t_j}  x(t) \: d t  & \approx &  \frac{h}{p} \: \sum_{\nu =1}^p 	\sum_{q =1}^\rho w_q \:  x{\left( t_{j-1} + \mu_{q \nu} h \right)}.
		\end{eqnarray}
		We prove the following lemma which will be used to find an asymptotic error expansion for the discrete orthogonal projection.
		\begin{lemma}\label{lem:1}
			Let $L_{\eta}$ be the Legendre polynomial of degree $\eta \in \left\{ 0, 1, \dots, r-1 \right\}$ defined on $[0, 1]$. Then for any $k = 1, 2, \dots, 2r-1$,
			\begin{eqnarray}\nonumber
				\frac{1}{p}  \: \sum_{\eta = 0}^{r-1} \sum_{\nu =1}^p 	\sum_{q =1}^\rho w_q \: L_{\eta}(\mu_{q \nu})  L_{\eta}(\tau) \frac{\left( \mu_{q \nu} - \tau  \right)^{k}}{k!} \: = \: \int_{0}^{1} \Lambda_{r}(\tau, s) \frac{\left( s - \tau  \right)^{k}}{k!} \: ds,
			\end{eqnarray}
			where $\displaystyle{ \sum_{\eta = 0}^{r-1} L_{\eta}(\tau) L_{\eta}(s) \: = \: \Lambda_r(\tau, s)}$, for $\tau, s \in [0, 1]$.
		\end{lemma}
		\begin{proof}
			Since $L_{\eta}$ is a polynomial of degree $0 \leq \eta \leq r-1$, 
			\begin{align*}
				\frac{1}{p}  \: \sum_{\nu =1}^p 	\sum_{q =1}^\rho w_q \:  L_\eta {\left(  \frac{\nu-1 + \mu_q}{p} \right)} & ~ = ~	\frac{1}{p}  \: \sum_{\nu =1}^p \int_{0}^{1} L_\eta {\left(  \frac{\nu-1 + t}{p} \right)} \: dt \\
				& = ~ \sum_{\nu =1}^p \int_{\frac{\nu-1}{p}}^{\frac{\nu}{p}} L_\eta(s) \: ds \\
				& = ~ \int_{0}^{1} L_\eta(s) \: ds.
			\end{align*}
			Since the basic quadrature formula \eqref{eq:Num_int} is exact for polynomials of degree $\leq 3r$,
			\begin{align*}
				\frac{1}{p}  \: \sum_{\nu =1}^p 	\sum_{q =1}^\rho w_q \:  L_\eta {\left(  \frac{\nu-1 + \mu_q}{p} \right)}  \left(  \frac{\nu-1 + \mu_q}{p} - \tau \right)^k ~ = ~ \int_{0}^{1} L_\eta(s) \left(  s - \tau \right)^k \: ds.
			\end{align*}
			It follows that
			\begin{align*}
				\frac{1}{p}  \: \sum_{\nu =1}^p 	\sum_{q =1}^\rho w_q \:  L_\eta {(\mu_{q \nu})} \: \frac{\left( \mu_{q \nu} - \tau \right)^k}{k!}  ~ = ~ \int_{0}^{1} L_\eta(s) \: \frac{\left(  s - \tau \right)^k}{k!}  \: ds,
			\end{align*}
			where $\displaystyle{\mu_{q \nu} = \frac{\nu-1 + \mu_q}{p}}.$ This gives
			{\small \begin{align*}
					\frac{1}{p}  \: \sum_{\eta = 0}^{r-1} \sum_{\nu =1}^p 	\sum_{q =1}^\rho w_q \: L_\eta(\tau) L_\eta(\mu_{q \nu}) \: \frac{\left( \mu_{q \nu} - \tau \right)^k}{k!} ~ = ~  \sum_{\eta = 0}^{r-1} L_\eta(\tau) \int_{0}^{1} L_\eta(s) \: \frac{\left(  s - \tau \right)^k}{k!} \: ds.
			\end{align*}}
			Let $$ \sum_{\eta = 0}^{r-1} L_{\eta}(\tau) L_{\eta}(s) \: = \: \Lambda_r(\tau, s), \quad \tau, s \in [0, 1].$$
			Then,
			\begin{align*}
				\frac{1}{p}  \: \sum_{\eta = 0}^{r-1} \sum_{\nu =1}^p 	\sum_{q =1}^\rho w_q \: L_\eta(\tau) L_\eta(\mu_{q \nu}) \: \frac{\left( \mu_{q \nu} - \tau \right)^k}{k!} ~ = ~   \int_{0}^{1} \Lambda_r(\tau, s) \: \frac{\left(  s - \tau \right)^k}{k!} \: ds.
			\end{align*}
			Hence the required result follows.
		\end{proof}
		
		\subsection{Discrete Orthogonal Projection}
		Let $j \in \left\{ 1, 2, \dots, n\right\}$ and $x, y \in C (\Delta_j)$. Define a discrete inner product on $\Delta_j$ by
		\begin{equation}\label{eq:dis_ip}
			\inp {x} {y}_{\Delta_j, m} = \tilde {h} \sum_{\nu = 1}^p \sum_{q = 1}^{\rho}   w_q \: x{\left (t_{j-1} + \mu_{q \nu} h \right )} \: y{\left( t_{j-1} + \mu_{q \nu} h \right )}.
		\end{equation}
		Note that, this is an indefinite inner product. For more details on indefinite inner product spaces, see \cite{Bognar}. However, the properties which we need to define a discrete orthogonal projection, hold true for \eqref{eq:dis_ip}.
		For $\eta = 0, 1, \ldots, r-1,$ let $L_\eta $ denote the Legendre polynomial of degree $\eta$ on $[0, 1].$ 
		For $j = 2, \ldots, n,$ and for  $\eta = 0, 1, \ldots, r-1,$ define
		\begin{eqnarray}\nonumber
			\varphi_{j, \eta} (t) &=& \left\{ {\begin{array}{ll}
					\sqrt {\frac {1} { h}} L_\eta \left ( \frac { t - t_{j-1}} {h} \right ), \;\;\;  t \in (t_{j-1}, t_j],   \\
					0, \;\;\;  \mbox{otherwise} 
			\end{array}}\right. 
		\end{eqnarray}
		and, $\varphi_{1, \eta} (t) = \sqrt {\frac {1} { h}} L_\eta \left ( \frac {t  - t_0} {h} \right ) $ if $t \in [t_{0}, t_1]$ and  
		$0$ otherwise. 
		Note that
		\begin{equation}\label{eqn:6.10}
			\varphi_{j, \eta} \left( t_{j-1} + \mu_{q \nu} h \right) = h^{- \frac{1}{2}} L_\eta(\mu_{q \nu}) \quad \text{ for all } j = 1, 2, \dots, n.
		\end{equation}
		Note that $\displaystyle {\{\varphi_{j, \eta} : ~  j = 1, \ldots, n, \; \eta = 0, 1,  \ldots, r-1 \}}$ be a set of orthonormal basis for  $\mathcal{X}_n,$ where $\varphi_{j, \eta}$ is the Legendre polynomial of degree $\eta$ defined on $[t_{j-1}, t_j]$. Since the basic numerical integration \eqref{eq:Num_int} has degree of precision $3r$, the set $\left\{\varphi_{j, \eta} \right\}$ is also orthonormal with respect to the discrete inner product \eqref{eq:dis_ip}. Let $\mathcal{P}_{r, \Delta_j}$ be the space of polynomials of degree $\leq r-1$ on $\Delta_j$.	Define the discrete orthogonal projection $P_{n,j}: C [t_{j-1}, t_j] \rightarrow \mathcal{P}_{r, \Delta_j}$ as follows:
		\begin{equation}\label{eq:discrete_proj1}
			P_{n, j} x = \sum_{\eta = 0}^{r-1} \inp {x} {\varphi_{j, \eta}}_{\Delta_j} \varphi_{j, \eta}. 
		\end{equation}
		See \cite{Atk-Pot2}, \cite{Atk-Pot3} for more details.
		A discrete orthogonal projection $P_n:  C[0, 1] \rightarrow \mathcal{X}_n$ is defined by
		\begin{eqnarray}\label{dop}
			P_n x = \sum_{j=1}^n P_{n, j} x.
		\end{eqnarray}
		It follows that $P_n x(t) = P_{n, j} x(t),$ for all $t \in [t_{j-1}, t_j]$. We also have the following error bound:\\
		$\norm{P_n} < \infty$ and also, if $x \in C^{r}[t_{j-1}, t_j]$, then
		\begin{equation}\label{proj:error1}
			\norm{x - P_{n, j}x}_{\Delta_j, \infty} \leq C_1 \norm{x^{(r)}}_{\Delta_j, \infty} h^r,
		\end{equation}
		if $x \in C^{r}[0, 1]$, then
		\begin{equation}\label{proj:error}
			\norm{x - P_{n}x}_{\Delta_j, \infty} \leq C_1 \norm{x^{(r)}}_{\infty} h^r,
		\end{equation}
		where $\displaystyle{\norm{x}_{\Delta_j, \infty} = \sup_{t \in [t_{j-1}, t_j]} \vert  x(t) \vert}$ and, $C_1$ is a constant independent of $h$. For details see \cite{RPK-GR3}. 
		
		In \eqref{proj:error} we have a error bound for the discrete orthogonal projection. But, by the following lemma we obtain an asymptotic error expansion for the discrete orthogonal projection, which is more stronger result than \eqref{proj:error}.

		\begin{lemma}\label{lem:2}
			Let $P_n$ be the discrete orthogonal projection defined by \eqref{eq:discrete_proj1} - \eqref{dop}. Let $\displaystyle{x \in C^{2r+2}_{\Delta^{(n)}}[0, 1]}$ and $t = t_{j-1} + \tau h$ with $\tau \in [0, 1]$. Then
			\begin{equation}\nonumber
				P_n x (t) - x(t) = \: \sum_{k = 1}^{2r+1} \: J_k(\tau) \: x^{(k)}(t_{j-1}+\tau h) \: h^k ~ + O\left( h^{2r+2} \right),
			\end{equation}
			where
			$\displaystyle{J_k(\tau) = \int_{0}^{1} \Lambda_{r}(\tau, s) \frac{\left( s - \tau  \right)^{k}}{k!} \: ds}, \quad k = 1, 2, \dots, 2r+1.$
		\end{lemma}
		
		\begin{proof}
			
			Define a function $v_j : [t_{j-1}, t_j] \rightarrow \mathbb{R}$ by
			$$v_j(t) = 1 , \quad t \in [t_{j-1}, t_j].$$
			For $\tau \in [0, 1]$, let $t = t_{j-1} + h \tau$ $\in [t_{j-1}, t_j].$ From \eqref{eq:discrete_proj1} it is easy to see that 
			$$P_{n, j} v_j = v_j.$$
			It follows that
			\begin{equation}\nonumber
				\sum_{\eta = 0}^{r-1} \inp {v_j} {\varphi_{j, \eta}}_{\Delta_j} \varphi_{j, \eta}(t) = 1.
			\end{equation}
			Since $\varphi_{j, \eta}$ is a polynomial of degree $0 \leq \eta \leq r-1$ on $[t_{j-1}, t_j],$
			\begin{eqnarray}\nonumber
				\inp {v_j} {\varphi_{j, \eta}}_{\Delta_j} =	\int_{t_{j -1}}^{t_j}  \varphi_{j, \eta}(s) \: d s  & = &  \frac{h}{p} \: \sum_{\nu =1}^p 	\sum_{q =1}^\rho w_q \:  \varphi_{j, \eta}{\left( t_{j-1} + \mu_{q \nu} h \right)}.
			\end{eqnarray}
			Thus for any function $x : [t_{j-1}, t_j] \rightarrow \mathbb{R}$, we have
			\begin{equation}\nonumber
				x(t) =  x(t)\:	\frac{h}{p} \: \sum_{\eta = 0}^{r-1} \sum_{\nu =1}^p 	\sum_{q =1}^\rho w_q \:  \varphi_{j, \eta}{\left( t_{j-1} + \mu_{q \nu} h \right)} \:  \varphi_{j, \eta}(t). 
			\end{equation}
			It follows that
			\begin{multline}\nonumber
				P_{n, j}x(t) - x(t) \\ = \frac{h}{p} \: \sum_{\eta = 0}^{r-1} \sum_{\nu =1}^p 	\sum_{q =1}^\rho w_q \:  \varphi_{j, \eta}{\left( t_{j-1} + \mu_{q \nu} h \right)} \left[ x(t_{j-1} + \mu_{q \nu} h) - x(t) \right]\:  \varphi_{j, \eta}(t),
			\end{multline}
			where $t = t_{j-1} + h \tau$ $\in [t_{j-1}, t_j]$ and $\tau \in [0, 1]$. From \eqref{eqn:6.10}, we have
			\begin{equation}\label{eq:3.15}
				P_{n, j}x(t) - x(t)  = \frac{1}{p} \: \sum_{\eta = 0}^{r-1} \sum_{\nu =1}^p 	\sum_{q =1}^\rho w_q \:  L_\eta(\mu_{q \nu})  \left[ x(t_{j-1} + \mu_{q \nu} h) - x(t) \right] \: L_\eta(\tau).
			\end{equation}
			Since $x \in C^{2r+2}[t_{j-1}, t_j]$, using Taylor series expansion we obtain 
			\begin{equation*}
				x(t_{j-1} + \mu_{q \nu} h) - x(t_{j-1} + h \tau) = \sum_{k = 1}^{2r+1} x^{(k)}(t_{j-1}+\tau h) \: \frac{\left( \mu_{q \nu} - \tau \right)^k}{k!} \: h^k ~ + O\left( h^{2r+2} \right).
			\end{equation*}
			Thus
			\begin{align*}
				P_{n} & x(t) -  x(t) = P_{n, j}x(t) - x(t)  \\ & =   \sum_{k = 1}^{2r+1} \: x^{(k)}(t_{j-1}+\tau h) \: h^k  \left\{ \frac{1}{p}  \sum_{\eta = 0}^{r-1} \sum_{\nu =1}^p 	\sum_{q =1}^\rho  w_q   L_\eta(\mu_{q \nu})   L_\eta(\tau)    \frac{\left( \mu_{q \nu} - \tau \right)^k}{k!}  \right\}  \\ & ~ +  O\left( h^{2r+2} \right).
			\end{align*}
			Let $$J_k(\tau) = \frac{1}{p} \: \sum_{\eta = 0}^{r-1} \sum_{\nu =1}^p 	\sum_{q =1}^\rho  w_q \:  L_\eta(\mu_{q \nu})  \: L_\eta(\tau)   \: \frac{\left( \mu_{q \nu} - \tau \right)^k}{k!}, \quad \tau \in [0, 1].$$  
			By Lemma \ref{lem:1}, we can write
			$J_k(\tau) = \int_{0}^{1} \Lambda_{r}(\tau, s) \frac{\left( s - \tau  \right)^{k}}{k!} \: ds. $ 
			Hence
			\begin{equation}\nonumber
				P_n x (t) - x(t) = \: \sum_{k = 1}^{2r+1} \: J_k(\tau) \: x^{(k)}(t_{j-1}+\tau h) \: h^k ~ + O\left( h^{2r+2} \right).
			\end{equation}
			The result follows.
		\end{proof}
		Let
		$$ \mathcal{L} = \left( I - \K'(\varphi) \right)^{-1} \K'(\varphi).$$
		Then $\mathcal{L}$ is a compact linear integral operator with kernel $\tilde{\ell}$. Note that the smoothness of $\tilde{\ell}$ is same as the kernel $\ell$. See \cite{riesz2012functional}, \cite[Lemma 5.1]{Atk-Pot} for details. It follows that
		\begin{equation*}
			\mathcal{L} P_n x (s) -  \mathcal{L} x (s)  = \sum_{j=1}^{n} \int_{t_{j-1}}^{t_j} \tilde{\ell}(s, t) \left( P_n x(t) - x(t) \right) dt \quad \forall x \in \mathcal{X}.	
		\end{equation*}
		Then using Lemma \ref{lem:2}, and following the proofs of \cite[Theorem 5.1]{McLean} and \cite[Theorem 3.2]{rakshit2020asymptotic}, it can be shown that
		\begin{equation}\label{asy_exp1}
			\mathcal{L} (I - P_n) \varphi (t_i) = \mathcal{E}_{2r}{\left( \varphi \right)}(t_i) h^{2r} + O \left( h^{2r+2} \right), \quad i = 0, 1, \ldots, n,
		\end{equation}
		where
		{\small \begin{multline*}
				\mathcal{E}_{2r}{\left( \varphi \right)}(t_i)=\bar{b}_{2r,2r} \int_0^1 \tilde{\ell}(t_i,t) (t) ~ \varphi^{(2r)}(t) \: dt \\ +\sum_{p=1}^{2r-1}\bar{b}_{2r,p}  \Bigg\{
				\left[ \left( \frac{\partial}{\partial t}\right) ^{2r-p - 1}\left( \tilde{\ell}(t_i,t) \varphi^{(p)}(t)\right) \right]_{t=0}^{t=1}  \\
				- \left[ \left( \frac{\partial}{\partial t}\right) ^{2r-p - 1}\left( \tilde{\ell}(t_i,t) \varphi^{(p)}(t)\right) \right]_{t=t_i-}^{t=t_i+} \Bigg\}
		\end{multline*}}
		with
		\begin{equation}\nonumber
			\bar{b}_{2r,p}= \int_{0}^{1} \int_{0}^{1} 
			\Lambda_{r}(\tau,s)\frac{(\tau-s)^{p}}{p!}\frac{B_{2r-p}(s)}{(2r-p)!} \: d\tau \: ds
		\end{equation}
		and $B_{k}$ is the Bernoulli polynomial of degree $k \geq 0$.

		\subsection{Approximation of the Integral Operator}
		Let $x \in \X$. Recall that
		\begin{equation*}
			\mathcal{K}({x})(s)= \int_{0}^{1} \kappa(s, t, x(t)) ~ dt, \quad s\in [0,1].
		\end{equation*}
		Replacing the above integral by the numerical quadrature rule (\ref{eq:Comp_Num_int}), we define the Nystr\"{o}m approximation of $\mathcal{K}$ by
		\begin{equation}\nonumber
			\mathcal{K}_m (x) (s)  =  \frac{h}{p} \sum_{j=1}^n  \sum_{q=1}^\rho \sum_{\nu =1}^p w_q \;  \kappa { \left( s,   t_{j-1} + \mu_{q \nu} h, x{\left( t_{j-1} + \mu_{q \nu} h \right)} \right) },  \quad  s \in [0, 1].
		\end{equation}
		Let $\{ \mu_{q \nu}^j = t_{j-1} + \mu_{q \nu} h : j = 1, 2, \dots, n; q = 1, 2, \dots, \rho; \nu = 1, 2, \dots, p \}$ denotes the set of all quadrature nodes in $[0, 1]$. Then
		\begin{equation}\nonumber
			\mathcal{K}_m (x) (s)  =  \frac{h}{p} \sum_{j=1}^n  \sum_{q=1}^\rho \sum_{\nu =1}^p w_q \;  \kappa { \left( s,   \mu_{q \nu}^j, x{\left( \mu_{q \nu}^j \right)} \right) },  \quad  s \in [0, 1].
		\end{equation}
		The Nystr\"{o}m method for solving \eqref{eq:main} is to find the element $x_m$ for which
		\begin{equation*}
			x_m - \mathcal{K}_m (x_m) = f.
		\end{equation*}
		For sufficiently large $m$, the above equation has a unique solution $\varphi_m$ in a neighborhood $B(\varphi, \epsilon)$ of $\varphi$, and
		\begin{eqnarray}\label{Gr_Ny_error}
			\|\varphi - \varphi_m \|_\infty   &\leq & C_2 \norm{\mathcal{K} (\varphi) - \mathcal{K}_m (\varphi)}_\infty \: = \: O \left (\tilde{h}^2 \right ),
		\end{eqnarray}
		where $C_2$ is a constant independent of $m$. See \cite[Theorem 4]{Atk_Ny}. We write
		\begin{align*}
			\left(I - P_n \right)\varphi_m = \left(I - P_n \right)\left( \varphi_m - \varphi \right) + \left(I - P_n \right)\varphi.
		\end{align*}
		Then from \eqref{proj:error}, \eqref{Gr_Ny_error}, we have
		\begin{equation}\label{eq:proj_error1}
			\left(I - P_n \right)\varphi_m =  O\left( \max\left\{ h^{r}, \tilde{h}^2 \right\}  \right).
		\end{equation}
		Let  $v_1, v_2 \in \X$ and $x \in B(\varphi, \epsilon)$. Then the Fr\'echet derivatives of $\mathcal{K}_m$ at $x$ are given by
		\begin{equation}\nonumber
			\mathcal{K}_m' (x) v_1(s)  =  \frac{h}{p} \sum_{j=1}^n  \sum_{q=1}^\rho \sum_{\nu =1}^p w_q \;  D^{(0,0,1)}\kappa{ \left( s,   \mu_{q \nu}^j, x{\left( \mu_{q \nu}^j \right)}\right)} v_1{\left( \mu_{q \nu}^j \right)},  \quad s \in [0, 1],
		\end{equation}
		%\begin{multline}\nonumber
		%	\mathcal{K}_m'' (\varphi) (v_1, v_2)(s)  \\ =  \frac{h}{p} \sum_{j=1}^n  \sum_{q=1}^\rho \sum_{\nu =1}^p w_q \;  D^{(0,0,2)}\kappa{ \left( s,   \mu_{q \nu}^j, \varphi{\left( \mu_{q \nu}^j \right)}\right)} v_1{\left( \mu_{q \nu}^j \right)}  v_2{\left( \mu_{q \nu}^j \right)} ,  \quad  s \in [0, 1].
		%\end{multline}
		%For any $v_1, v_2 \in \X$, we have
		\begin{equation}\nonumber
			\mathcal{K}_m'' (x)\left( v_1, v_2 \right)(s)  =  \frac{h}{p} \sum_{j=1}^n  \sum_{q=1}^\rho \sum_{\nu =1}^p w_q \;  \frac{\partial^2 \kappa}{\partial u^2} { \left( s,   \mu_{q \nu}^j, x{\left( \mu_{q \nu}^j \right)}\right)} v_1{\left( \mu_{q \nu}^j \right)} v_2{\left( \mu_{q \nu}^j \right)}.
		\end{equation}
		It follows that
		\begin{equation*}
			\norm{\mathcal{K}_m'' (x)\left( v_1, v_2 \right)}_\infty \leq \left( \sup_{\stackrel {s, t \in [0, 1]}{\lvert u \rvert \leq \|\varphi \|_\infty + \epsilon}} 	\left\lvert  \frac{\partial^2 \kappa}{\partial u^2}(s, t, u) \right\rvert \right) \norm{v_1}_\infty \norm{v_2}_\infty.
		\end{equation*}
		This implies 
		\begin{equation*}
			\norm{\mathcal{K}_m'' (x)}  < \infty
		\end{equation*}
		Similarly, it can be shown that
		{\small \begin{equation*}
				\norm{\mathcal{K}_m^{(3)} (x)} \leq \left( \sup_{\stackrel {s, t \in [0, 1]}{\lvert u \rvert \leq \|\varphi \|_\infty + \epsilon}} 	\left\lvert  \frac{\partial^3 \kappa}{\partial u^3}(s, t, u) \right\rvert \right) = C_3 < \infty
		\end{equation*}}
		\begin{lemma}\label{lip_2}
			Let $x_1, x_2 \in B(\varphi, \epsilon)$. If $D^{(0, 0, 3)}\kappa \in C(\Omega)$ then
			\begin{equation*}
				\norm{\mathcal{K}_m'' (x_1) - \mathcal{K}_m'' (x_2)} \leq C_3 \norm{x_1 - x_2}_\infty,
			\end{equation*}
			where $C_3$ is constant independent of $n$.
			\begin{proof}
				For $v_1, v_2 \in \X$, we have
				{\footnotesize \begin{multline*}
						\left[ \mathcal{K}_m'' (x_1) - \mathcal{K}_m'' (x_2) \right](v_1, v_2)(s) \\ 
						= \frac{h}{p} \sum_{j=1}^n  \sum_{q=1}^\rho \sum_{\nu =1}^p w_q \left[ \frac{\partial^2 \kappa}{\partial u^2}{ \left( s,   \mu_{q \nu}^j, x_1{\left( \mu_{q \nu}^j \right)}\right)} - \frac{\partial^2 \kappa}{\partial u^2}{ \left( s,   \mu_{q \nu}^j, x_2{\left( \mu_{q \nu}^j \right)}\right)} \right] v_1{\left( \mu_{q \nu}^j \right)}  v_2{\left( \mu_{q \nu}^j \right)}
				\end{multline*}}
				for all $s \in [0, 1]$. Since $D^{(0, 0, 3)}\kappa \in C(\Omega)$, applying mean value theorem on $\frac{\partial^2 \kappa}{\partial u^2}$ with respect to its third variable $u$, we obtain
				\begin{multline*}
					\frac{\partial^2 \kappa}{\partial u^2}{ \left( s,   \mu_{q \nu}^j, x_1{\left( \mu_{q \nu}^j \right)}\right)} - \frac{\partial^2 \kappa}{\partial u^2}{ \left( s,   \mu_{q \nu}^j, x_2{\left( \mu_{q \nu}^j \right)}\right)}  = \left[ x_1{\left( \mu_{q \nu}^j \right)} -x_2{\left( \mu_{q \nu}^j \right)} \right] \frac{\partial^3 \kappa}{\partial u^3}{ \left( s,   \mu_{q \nu}^j,  \zeta_{q \nu}^j \right)},
				\end{multline*} 
				where $\zeta_{q \nu}^j$ lies in the line segment joining the points $x_1{\left( \mu_{q \nu}^j \right)}$ and $x_2{\left( \mu_{q \nu}^j \right)}$. Then
				\begin{eqnarray*}
					\left \lvert  \frac{\partial^2 \kappa}{\partial u^2}{ \left( s,   \mu_{q \nu}^j, x_1{\left( \mu_{q \nu}^j \right)}\right)} - \frac{\partial^2 \kappa}{\partial u^2}{ \left( s,   \mu_{q \nu}^j, x_2{\left( \mu_{q \nu}^j \right)}\right)}  \right \rvert  \leq C_3 \norm{x_1 - x_2}_\infty.
				\end{eqnarray*}
				Hence
				\begin{align*}
					\norm{\left[ \mathcal{K}_m'' (x_1) - \mathcal{K}_m'' (x_2) \right](v_1, v_2)}_\infty \leq C_3 \norm{x_1 - x_2}_\infty \norm{v_1}_\infty \norm{v_2}_\infty,
				\end{align*}
				which follows the result.
			\end{proof}
		\end{lemma}
		
		We will now quote some error estimates for the Nystr\"om approximations. 
		
		For $\alpha \geq 0$, if $v_1, v_2 \in C_{\Delta^{(m)}}^{\alpha}[0, 1]$, then from \cite{schumaker} or \cite[Corollary 1]{Atk-Pot2}, we obtain the following errors for numerical integration.
		\begin{eqnarray}\label{eq:Ny2}
			\norm{\left[ {\mathcal{K}}_m'(\varphi) - {\mathcal{K}}'(\varphi) \right]v_1}_\infty = O\left(\tilde{h}^2\right),
		\end{eqnarray}
		\begin{eqnarray*}
			\norm{\left[ {\mathcal{K}}_m''(\varphi) - {\mathcal{K}}''(\varphi) \right]\left( v_1, v_2 \right)}_\infty = O\left(\tilde{h}^2\right).
		\end{eqnarray*}
		Also from \cite[Proposition 3.3]{RPK-GR1}, we have
		%	\begin{eqnarray}\label{eq:res1}
			%		\norm{{\mathcal{K}}_m'(x) - {\mathcal{K}}_m'(y)} \leq C \norm{x-y}_{\infty} \quad \text{for all } x, y \in B\left(\varphi, \epsilon\right).
			%	\end{eqnarray}
		%	It follows that
		\begin{eqnarray}\label{eq:Ny3}
			\norm{{\mathcal{K}}_m'(\varphi_m) - {\mathcal{K}}_m'(\varphi)} \leq C_4 \norm{\varphi_m - \varphi}_{\infty} = O\left(\tilde{h}^2\right).
		\end{eqnarray}
		Therefore combining \eqref{eq:Ny2} and the above equation, we obtain
		\begin{eqnarray}\label{eq:Ny4}
			\norm{\left[ {\mathcal{K}}_m'(\varphi_m) - {\mathcal{K}}'(\varphi) \right]v}_\infty = O\left(\tilde{h}^2\right), \quad \text{ for all } v \in C_{\Delta_m}^\nu[0, 1].
		\end{eqnarray}
		Similarly,
		\begin{eqnarray}\label{eq:Ny5}
			\norm{\left[ {\mathcal{K}}_m''(\varphi_m) - {\mathcal{K}}''(\varphi) \right]\left(v_1, v_2\right)}_\infty = O\left(\tilde{h}^2\right), \quad \forall v_1, v_2 \in C_{\Delta_m}^\nu[0, 1],
		\end{eqnarray}
		for all $ v_1, v_2, v_3 \in C_{\Delta_m}^\nu[0, 1]$ implies
		\begin{eqnarray}\label{eq:Ny35}
			\norm{\left[ {\mathcal{K}}_m^{(3)}(\varphi_m) - {\mathcal{K}}^{(3)}(\varphi) \right]\left(v_1, v_2, v_3\right)}_\infty = O\left(\tilde{h}^2\right).
		\end{eqnarray}	
		\section{Asymptotic Error Analysis}
		Replacing $\mathcal{K}$ by $\mathcal{K}_m$ and $\pi_n$ by $P_n$ in the Galerkin equation $x - \pi_n \mathcal{K}(x) = \pi_n f$, the discrete Galerkin equation is defined by $z_n^G - P_n\mathcal{K}_m (z_n^G) = P_n f$, where $z_n^G$ is the discrete Galerkin solution. Then the discrete iterated Galerkin solution is defined by
		\begin{equation*}
			z_n^S = \mathcal{K}_m (z_n^G) +f.
		\end{equation*} 
		Note that $P_n z_n^S = z_n^G$.
		From the equations
		$ \varphi_m - \mathcal{K}_m(\varphi_m) = f$ and $z_n^S - \mathcal{K}_m(P_n z_n^S) = f$,
		we obtain the following error term.
		\begin{align}\label{equation:1}
			z_n^S - \varphi_m = & \left[ I - {\mathcal{K}}_m'(\varphi_m) \right]^{-1} \left[\mathcal{K}_m(z_n^G) -  \mathcal{K}_m(\varphi_m) - {\mathcal{K}}_m'(\varphi_m) (z_n^G - \varphi_m) \right] \nonumber\\
			& - \mathcal{L}_m (I -P_n) \left[\mathcal{K}_m(z_n^G) -  \mathcal{K}_m(\varphi_m) - {\mathcal{K}}_m'(\varphi_m) (z_n^G - \varphi_m) \right] \nonumber\\
			& - \mathcal{L}_m (I -P_n){\mathcal{K}}_m'(\varphi_m) (z_n^G - \varphi_m) \nonumber\\
			& - \mathcal{L}_m (I -P_n) \varphi_m,
		\end{align}
		where$$\mathcal{L}_m = \left[ I - {\mathcal{K}}_m'(\varphi_m) \right]^{-1} {\mathcal{K}}_m'(\varphi_m).$$	
		
		Using the Resolvent Identity, we get
		\begin{multline*}
			\left( I - {\mathcal{K}}_m'(\varphi_m) \right)^{-1}  -  \left( I - {\mathcal{K}}'(\varphi) \right)^{-1} \\
			%	 &  =  \left( I - {\mathcal{K}}'(\varphi) \right)^{-1} \left[ \left( I - {\mathcal{K}}'(\varphi) \right) - \left( I - {\mathcal{K}}_m'(\varphi_m) \right) \right] \left( I - {\mathcal{K}}_m'(\varphi_m)  \right)^{-1} \\
			=  \left( I - {\mathcal{K}}'(\varphi) \right)^{-1} \left[  {\mathcal{K}}_m'(\varphi_m) - {\mathcal{K}}'(\varphi) \right] \left( I - {\mathcal{K}}_m'(\varphi_m)  \right)^{-1}
			%	& = \left( I - {\mathcal{K}}'(\varphi) \right)^{-1} \left[  {\mathcal{K}}_m'(\varphi_m) - {\mathcal{K}}_m'(\varphi) + {\mathcal{K}}_m'(\varphi) - {\mathcal{K}}'(\varphi) \right] \left( I - {\mathcal{K}}_m'(\varphi_m)  \right)^{-1}
		\end{multline*}
		Therefore
		\begin{align}\label{eq:res-id}
			\left( I - {\mathcal{K}}_m'(\varphi_m) \right)^{-1} & = \left( I - {\mathcal{K}}'(\varphi) \right)^{-1} \nonumber \\ 
			& ~~~ + \left( I - {\mathcal{K}}'(\varphi) \right)^{-1} \left[  {\mathcal{K}}_m'(\varphi_m) - {\mathcal{K}}_m'(\varphi)  \right] \left( I - {\mathcal{K}}_m'(\varphi_m)  \right)^{-1} \nonumber \\ 
			& ~~~ + \left( I - {\mathcal{K}}'(\varphi) \right)^{-1} \left[ {\mathcal{K}}_m'(\varphi) - {\mathcal{K}}'(\varphi) \right] \left( I - {\mathcal{K}}_m'(\varphi_m)  \right)^{-1}
		\end{align}	
		
		Now, we will analyze each of the terms appearing in the RHS of the equation \eqref{equation:1}.	Error estimates for each of the said terms will be obtained by the following propositions. 
		
		\begin{proposition} \label{prop:3} Let $\left\{t_i : i = 0, 1, \ldots, n \right\}$ be the set of partition points of $[0, 1]$ defined by \eqref{eq:partition1}, then
			\begin{equation*}
				\mathcal{L}_m (I -P_n) \varphi_m (t_i) = \mathcal{E}_{2r}{\left( \varphi \right)}(t_i) h^{2r} + O \left( \max\left\{ h^{2r+2}, \tilde{h}^2 \right\} \right),
			\end{equation*}
			where $\mathcal{E}_{2r}$ is defined by \eqref{asy_exp1}.
		\end{proposition}
		\begin{proof} It can be easily verified that (using \eqref{eq:res-id})
			{\small	\begin{align}\label{eq:9.25}
					\mathcal{L}_m& (I -P_n) \varphi_m   =  \left[ I - {\mathcal{K}}_m'(\varphi_m) \right]^{-1} {\mathcal{K}}_m'(\varphi_m) (I -P_n) \varphi_m \nonumber\\
					& = \left( I - {\mathcal{K}}'(\varphi) \right)^{-1} {\mathcal{K}}_m'(\varphi_m) (I -P_n) \varphi_m \nonumber \\
					& ~~~ +  \left( I - {\mathcal{K}}'(\varphi) \right)^{-1} \left[  {\mathcal{K}}_m'(\varphi_m) - {\mathcal{K}}_m'(\varphi)  \right] \left( I - {\mathcal{K}}_m'(\varphi_m)  \right)^{-1} {\mathcal{K}}_m'(\varphi_m) (I -P_n) \varphi_m \nonumber \\
					& ~~~ + \left( I - {\mathcal{K}}'(\varphi) \right)^{-1} \left[ {\mathcal{K}}_m'(\varphi) - {\mathcal{K}}'(\varphi) \right] \left( I - {\mathcal{K}}_m'(\varphi_m)  \right)^{-1} {\mathcal{K}}_m'(\varphi_m) (I -P_n) \varphi_m.
			\end{align}}
			Consider the first term of the above equation, we have
			\begin{align*}
				\left( I - {\mathcal{K}}'(\varphi) \right)^{-1} & {\mathcal{K}}_m'(\varphi_m) (I -P_n) \varphi_m \\
				& = \left( I - {\mathcal{K}}'(\varphi) \right)^{-1} {\mathcal{K}}'(\varphi) (I -P_n) \varphi \\
				& ~~~ + \left( I - {\mathcal{K}}'(\varphi) \right)^{-1} {\mathcal{K}}'(\varphi) (I -P_n) (\varphi_m - \varphi) \\
				& ~~~ + \left( I - {\mathcal{K}}'(\varphi) \right)^{-1} \left[{\mathcal{K}}_m'(\varphi_m) - {\mathcal{K}}'(\varphi)\right] (I -P_n) \varphi_m.
			\end{align*}
			Using \eqref{asy_exp1}, \eqref{Gr_Ny_error} and \eqref{eq:Ny4}, we obtain
			\begin{multline}\label{eq:9.26}
				\left( I - {\mathcal{K}}'(\varphi) \right)^{-1}  {\mathcal{K}}_m'(\varphi_m) (I -P_n) \varphi_m (t_i) = \mathcal{E}_{2r}{\left( \varphi \right)}(t_i) h^{2r} + O \left( \max\left\{ h^{2r+2}, \tilde{h}^2 \right\} \right).
			\end{multline}
			Note that
			\begin{eqnarray*}
				\norm{	\mathcal {K}_m'  (\varphi_m)} \leq \sup_{\stackrel {s, t \in [0, 1]}{\lvert u \rvert \leq \|\varphi \|_\infty + \epsilon}}
				\lvert \kappa_u (s, t, u) \rvert
			\end{eqnarray*}
			and from \cite[Proposition 4.2]{RPK-GR1}, we have $\norm{ \left( I - {\mathcal{K}}_m'(\varphi_m)  \right)^{-1}} < \infty.$ Thus, from \eqref{eq:Ny2} and \eqref{eq:Ny3}, we have the followings
			{\small \begin{multline*}
					\norm{\left( I - {\mathcal{K}}'(\varphi) \right)^{-1} \left[  {\mathcal{K}}_m'(\varphi_m) - {\mathcal{K}}_m'(\varphi)  \right] \left( I - {\mathcal{K}}_m'(\varphi_m)  \right)^{-1} {\mathcal{K}}_m'(\varphi_m) (I -P_n) \varphi_m}_\infty   = O\left( \tilde{h}^2\right),
			\end{multline*}}
			{\small	\begin{multline*}
					\norm{\left( I - {\mathcal{K}}'(\varphi) \right)^{-1} \left[  {\mathcal{K}}_m'(\varphi_m) - {\mathcal{K}}_m'(\varphi)  \right] \left( I - {\mathcal{K}}_m'(\varphi_m)  \right)^{-1} {\mathcal{K}}_m'(\varphi_m) (I -P_n) \varphi_m}_\infty   = O\left( \tilde{h}^2\right).
			\end{multline*}}
			\noindent
			Hence the required result follows from \eqref{eq:9.25}, \eqref{eq:9.26} and the above two estimates.
		\end{proof}
		
		Before each of the following propositions, we prove lemmas and its corollaries which are used to prove next propositions.
		
		\begin{lemma}\label{lem:4}
			Let $P_n$ be the discrete orthogonal projection defined by \eqref{dop}. If $ D^{(0, 0, 3)} \kappa \in C(\Omega)$ and $v \in C^{r + 2 } ( [0, 1]), $ then for $r \geq 1$
			\begin{align} \label{eq:exp2}
				\mathcal {K}'' (\varphi)  ( P_n v - v)^2   =  {T} ( v ) h^{ 2 r } + O (h^{2 r + 2}), 
			\end{align}
			where
			\begin{equation*}
				{T} ( v ) = \left ( \int_0^1 J_r (\tau)^2 \: d \tau \right ) \mathcal {K}'' (\varphi)   
				\left (v^{(r)} \right )^2.
			\end{equation*}
			Furthermore, when $r = 1$, then 
			\begin{equation}\label{eq:exp4}
				\mathcal {K}^{(3)} (\varphi)  ( P_n v - v)^3   =
				O \left( h^{4} \right).
			\end{equation}
		\end{lemma}
		\begin{proof}
			The proofs of \eqref{eq:exp2} and \eqref{eq:exp4} follows from Lemma \ref{lem:2}, \cite[Lemma 2.4]{RPK-Nidhin} and \cite[Remark 2.4]{RPK-Nidhin} respectively.
		\end{proof}
		Given that $\left( I - {\mathcal{K}}'(\varphi) \right)^{-1}$ is a bounded linear operator. Let $$\mathcal{M} = \left( I - {\mathcal{K}}'(\varphi) \right)^{-1} \mathcal {K}'' (\varphi).$$ Note that $\mathcal{M}$ is a compact bi-linear integral operator. Also the smoothness of the kernel of $\mathcal{M}$, is same as the kernels of $\mathcal {K}'' (\varphi)$.  See \cite{Atk-Pot2}, \cite{riesz2012functional}.
		
		As a consequence of the above lemma, we get the following result.
		{\small \begin{Corollary} For $r \geq 1$,
				\begin{align} \label{eq:exp21}
					\mathcal{M}( P_n v - v)^2   =  \mathcal{T} ( v ) h^{ 2 r } + O (h^{2 r + 2}), 
				\end{align}
				where
				\begin{equation*}
					\mathcal{T} ( v ) = \left ( \int_0^1 J_r (\tau)^2 \: d \tau \right ) 	\mathcal{M}   \left (v^{(r)} \right )^2.
				\end{equation*}
			\end{Corollary}
		}
		\begin{lemma}\label{lem:5} 
			If $\varphi \in C^{r + 2 } ( [0, 1])$, then for $r \geq 1$, $$\K_m''(\varphi_m)(z_n^G - \varphi_m)^2 = {T} ( \varphi ) h^{ 2 r } + O \left(\max\left\{ h^{2 r + 2}, \tilde{h}^2 \right\}\right),$$
			where $T$ is defined in Lemma \ref{lem:4}.
		\end{lemma}
		\begin{proof}Note that
			\begin{align}\label{eq:4.32}
				z_n^G - \varphi_m & = P_n z_n^S - P_n \varphi_m - \varphi_m + P_n \varphi_m + \varphi \notag\\
				& = P_n \left(  z_n^S - \varphi_m \right) - \left( I - P_n \right)\varphi_m
			\end{align}
			Thus,
			\begin{align}\label{eq:4.29}
				\K_m''& (\varphi_m)(z_n^G - \varphi_m)^2 \notag\\
				& = \K_m''(\varphi_m)\left( P_n \left(  z_n^S - \varphi_m \right) \right)^2 - 2 \K_m''(\varphi_m)\left( P_n \left(  z_n^S - \varphi_m \right), \left( I - P_n \right)\varphi_m \right) \notag\\ 
				& ~~~ + \K_m''(\varphi_m)\left( \left( I - P_n \right)\varphi_m \right)^2 
			\end{align}
			Since $\norm{\K_m''(\varphi_m)} < \infty $ and $\norm{P_n} < \infty $, from \eqref{eq:discrete_Gal} it is easy to see that 
			\begin{equation}\label{eq:4.30}
				\norm{\K_m''(\varphi_m)\left( P_n \left(  z_n^S - \varphi_m \right) \right)^2}_\infty = O\left( \max\left\{ h^{2r+4}, \tilde{h}^4 \right\}  \right),
			\end{equation}
			\begin{equation}\label{eq:4.31}
				\norm{\K_m''(\varphi_m)\left( P_n \left(  z_n^S - \varphi_m \right), \left( I - P_n \right)\varphi_m \right)}_\infty = O\left( \tilde{h}^2  \max\left\{ h^{r+2}, \tilde{h}^2 \right\}  \right).
			\end{equation}
			Now we write
			{\small	\begin{align*}
					\K_m''(\varphi_m)\left( \left( I - P_n \right)\varphi_m \right)^2  & =  \left[ \K_m''(\varphi_m) -  \K''(\varphi) \right]\left( \left( I - P_n \right)\varphi_m \right)^2 \\
					& ~~~ + \K''(\varphi)\left( \left( I - P_n \right)\varphi_m \right)^2 \\ & =  \K''(\varphi)\left( \left( I - P_n \right)\varphi \right)^2 + \K''(\varphi)\left( \left( I - P_n \right) \left( \varphi - \varphi_m \right) \right)^2 \\
					& ~~~ + \left[ \K_m''(\varphi_m) -  \K''(\varphi) \right]\left( \left( I - P_n \right)\varphi_m \right)^2.
			\end{align*}}
			Since $\norm{\K_m''(\varphi_m)} < \infty $ and $\norm{P_n} < \infty $, from \eqref{Gr_Ny_error}, \eqref{eq:Ny5} and the above estimate, we obtain
			\begin{equation*}
				\K_m''(\varphi_m)\left( \left( I - P_n \right)\varphi_m \right)^2  = \K''(\varphi)\left( \left( I - P_n \right)\varphi \right)^2 + O\left(\tilde{h}^2 \right).
			\end{equation*}
			Therefore, from \eqref{eq:exp2} we obtain
			\begin{align*}
				\K_m''(\varphi_m)\left( \left( I - P_n \right)\varphi_m \right)^2  = {T} ( \varphi ) h^{ 2 r } + O \left(\max\left\{ h^{2 r + 2}, \tilde{h}^2 \right\}\right).
			\end{align*}
			Hence, the required result follows from \eqref{eq:4.29}, \eqref{eq:4.30}, \eqref{eq:4.31} and the above equation.
		\end{proof}
		From the above lemma, we obtain 
		\begin{equation}\label{eq:lem61}
			\left( I - {\mathcal{K}}'(\varphi) \right)^{-1}\K_m''(\varphi_m)(z_n^G - \varphi_m)^2  = \mathcal{T} ( \varphi ) h^{ 2 r } + O \left(\max\left\{ h^{2 r + 2}, \tilde{h}^2 \right\}\right),
		\end{equation}
		where $\mathcal{T}$ is defined by \eqref{eq:exp21}.
		
		\begin{lemma}\label{lem:6}
			If $\varphi \in C^{r + 2 } ( [0, 1])$, then
			\begin{equation*}
				\left[ I - {\mathcal{K}}_m'(\varphi_m) \right]^{-1} \K_m^{(3)}(\varphi_m)(z_n^G - \varphi_m)^3 = \left\{ {\begin{array}{ll}
						O \left(\max\left\{ h^{4}, \tilde{h}^2 \right\}\right), \;\;\;  r = 1,   \vspace*{1.5mm}\\
						O \left(\max\left\{ h^{3r}, \tilde{h}^6 \right\}\right), ~  r \geq 2.
				\end{array}}\right.
			\end{equation*}
		\end{lemma}
		\begin{proof}
			First, we consider the case when $r \geq 2$. Since $\norm{\left[ I - {\mathcal{K}}_m'(\varphi_m) \right]^{-1}} < \infty$ and $\norm{\K_m^{(3)}(\varphi_m)} < \infty$, from \eqref{eq:discrete_Gal} we obtain
			\begin{equation*}
				\norm{\left[ I - {\mathcal{K}}_m'(\varphi_m) \right]^{-1} \K_m^{(3)}(\varphi_m)(z_n^G - \varphi_m)^3}_\infty = O \left(\max\left\{ h^{3r}, \tilde{h}^6 \right\}\right).
			\end{equation*}
			Now consider the case, when $r = 1$. We rewrite \eqref{eq:4.32} as $$ (z_n^G - \varphi_m)^3 =  \left[ P_n \left(  z_n^S - \varphi_m \right) - \left( I - P_n \right)\varphi_m \right]^3.$$
			Thus
			{\small	\begin{align}\label{eq:4.36}
					\K_m^{(3)}(\varphi_m)(z_n^G - \varphi_m)^3 & = \K_m^{(3)}(\varphi_m)\left(  P_n \left(  z_n^S - \varphi_m \right) \right)^3 - \K_m^{(3)}(\varphi_m)\left(  \left( I - P_n \right)\varphi_m \right)^3 \nonumber \\
					& ~~~ - \K_m^{(3)}(\varphi_m) \left( \left(  P_n \left(  z_n^S - \varphi_m \right) \right)^2, \left( I - P_n \right)\varphi_m \right) \nonumber\\
					& ~~~ + \K_m^{(3)}(\varphi_m) \left(  P_n \left(  z_n^S - \varphi_m \right) , \left( \left( I - P_n \right)\varphi_m \right)^2 \right).
			\end{align}}
			Since $\norm{\K_m^{(3)}(\varphi_m)} < \infty $ and $\norm{P_n} < \infty $, from \eqref{eq:discrete_it_Gal} and \eqref{eq:proj_error1} we obtain
			\begin{equation*}
				\K_m^{(3)}(\varphi_m)\left(  P_n \left(  z_n^S - \varphi_m \right) \right)^3 = O \left( h^{6} \right),
			\end{equation*}
			\begin{equation*}
				\K_m^{(3)}(\varphi_m) \left( \left(  P_n \left(  z_n^S - \varphi_m \right) \right)^2, \left( I - P_n \right)\varphi_m \right) = O \left( h^{5} \right),
			\end{equation*}
			\begin{equation*}
				\K_m^{(3)}(\varphi_m) \left(  P_n \left(  z_n^S - \varphi_m \right) , \left( \left( I - P_n \right)\varphi_m \right)^2 \right) \\ = O \left( h^{4} \right).
			\end{equation*}
			Note that, we have used the fact $\tilde{h} \leq h$ in the above three expressions. 
			On the other hand, from \eqref{Gr_Ny_error}, \eqref{eq:Ny35} we have
			\begin{equation*}
				\K_m^{(3)}(\varphi_m)\left(  \left( I - P_n \right)\varphi_m \right)^3  = \K^{(3)}(\varphi)\left(\left( I - P_n \right)\varphi \right)^3 + O\left(\tilde{h}^2 \right).
			\end{equation*}
			From \eqref{eq:exp4}, it follows that
			\begin{equation}\label{eq:4.37}
				\K_m^{(3)}(\varphi_m)\left(  \left( I - P_n \right)\varphi_m \right)^3  =  	 O \left(\max\left\{ h^{4}, \tilde{h}^2 \right\}\right).
			\end{equation}
			Now, combining the results \eqref{eq:4.36} - \eqref{eq:4.37}, we obtain
			\begin{equation*}
				\K_m^{(3)}(\varphi_m)(z_n^G - \varphi_m)^3 = O \left(\max\left\{ h^{4}, \tilde{h}^2 \right\}\right).
			\end{equation*}
			Therefore \begin{equation*}
				\left[ I - {\mathcal{K}}_m'(\varphi_m) \right]^{-1}\K_m^{(3)}(\varphi_m)(z_n^G - \varphi_m)^3 = O \left(\max\left\{ h^{4}, \tilde{h}^2 \right\}\right).
			\end{equation*}
			Hence follows the result.
		\end{proof}

		\begin{proposition}\label{prop:4}
			Let $\varphi \in C^{r + 2 } ( [0, 1])$. Then for $r \geq 1$,
			\begin{multline*}
				\left[ I - {\mathcal{K}}_m'(\varphi_m) \right]^{-1}  \left[\mathcal{K}_m(z_n^G) -  \mathcal{K}_m(\varphi_m) - {\mathcal{K}}_m'(\varphi_m) (z_n^G - \varphi_m) \right]	(s)\\
				= \frac{1}{2} \mathcal{T} ( \varphi )(s) h^{ 2 r } + O \left(\max\left\{ h^{2 r + 2}, \tilde{h}^2 \right\}\right),
			\end{multline*}
			for all $s \in [0, 1]$.
		\end{proposition}
		\begin{proof}
			Applying the generalized Taylor's series expansion of $\K_m$ about $\varphi_m$ in the neighbourhood ${B}(\varphi, \epsilon)$, we obtain
			\begin{multline}\label{eq:4.41}
				\mathcal{K}_m(z_n^G) -  \mathcal{K}_m(\varphi_m) - {\mathcal{K}}_m'(\varphi_m) (z_n^G - \varphi_m)  \\ 
				= \frac{1}{2} \K_m''(\varphi_m)(z_n^G - \varphi_m)^2 + \frac{1}{6} \K_m^{(3)}(\varphi_m)(z_n^G - \varphi_m)^3  + \mathcal{R}_{4, m} \left( z_n^G - \varphi_m \right),
			\end{multline}
			where
			$$\mathcal{R}_{4, m} \left( z_n^G - \varphi_m \right) =  \int_0^1 \frac{(1 - \theta)^3}{3!}\mathcal{K}_m^{(4)} \left(\varphi_m + \theta (z_n^G - \varphi_m) \right) (z_n^G - \varphi_m)^4
			\; d \theta.$$
			Note that for any $x \in {B}(\varphi, \epsilon)$, $v \in \X$,
			\begin{equation}\nonumber
				\mathcal{K}_m^{(4)} (x) v^4(s)  =  \frac{h}{p} \sum_{j=1}^n  \sum_{q=1}^\rho \sum_{\nu =1}^p w_q \;  \frac{\partial^4 \kappa}{\partial u^4} { \left( s,   \mu_{q \nu}^j, x{\left( \mu_{q \nu}^j \right)}\right)} v^4{\left( \mu_{q \nu}^j \right)}, \quad  s \in [0, 1].
			\end{equation}
			It follows that
			\begin{eqnarray*}
				\norm{\mathcal{K}_m^{(4)} (x) v^4}_\infty \leq \left( \sup_{\stackrel {s, t \in [0, 1]}{\lvert u \rvert \leq \|\varphi \|_\infty + \epsilon}} 	\left\lvert  \frac{\partial^4 \kappa}{\partial u^4}(s, t, u) \right\rvert \right) \norm{v}^4_\infty = C_5  \norm{v}^4_\infty.
			\end{eqnarray*}
			Since $\varphi_m$ and $z_n^G \in {B}(\varphi, \epsilon)$, $\varphi_m + \theta (z_n^G - \varphi_m) \in {B}(\varphi, \epsilon)$ and therefore
			\begin{eqnarray*}
				\norm{\mathcal{K}_m^{(4)} \left(\varphi_m + \theta (z_n^G - \varphi_m) \right) (z_n^G - \varphi_m)^4}_\infty \leq C_5 \norm{z_n^G - \varphi_m}_\infty^4 = O\left( h^{4r} \right).
			\end{eqnarray*}
			It follows that 
			\begin{equation*}
				\mathcal{R}_{4, m} \left( z_n^G - \varphi_m \right) = O\left( h^{4r} \right).
			\end{equation*}
			Using the resolvent identity \eqref{eq:Ny4} and \eqref{eq:res-id}, we obtain
			\begin{multline*}
				\left[ I - {\mathcal{K}}_m'(\varphi_m) \right]^{-1} \K_m''(\varphi_m)(z_n^G - \varphi_m)^2   = 	\left[ I - {\mathcal{K}}'(\varphi) \right]^{-1} \K_m''(\varphi_m)(z_n^G - \varphi_m)^2 + O \left(  \tilde{h}^2 \right).
			\end{multline*}
			By \eqref{eq:lem61}, it follows that
			\begin{equation*}
				\left[ I - {\mathcal{K}}_m'(\varphi_m) \right]^{-1} \K_m''(\varphi_m)(z_n^G - \varphi_m)^2  = \mathcal{T} ( \varphi ) h^{ 2 r } + O \left(\max\left\{ h^{2 r + 2}, \tilde{h}^2 \right\}\right).
			\end{equation*}
			From the Lemma \ref{lem:6}, we have
			\begin{equation}\label{eq:4.44}
				\left[ I - {\mathcal{K}}_m'(\varphi_m) \right]^{-1} \K_m^{(3)}(\varphi_m)(z_n^G - \varphi_m)^3 = \left\{ {\begin{array}{ll}
						O \left(\max\left\{ h^{4}, \tilde{h}^2 \right\}\right), \;\;\;  r = 1,   \vspace*{1.5mm}\\
						O \left(\max\left\{ h^{3r}, \tilde{h}^6 \right\}\right), ~  r \geq 2.
				\end{array}}\right.
			\end{equation}
			Combining the results from \eqref{eq:4.41} to \eqref{eq:4.44}, we obtain for $r \geq 1$,
			\begin{multline*}
				\left[ I - {\mathcal{K}}_m'(\varphi_m) \right]^{-1}  \left[\mathcal{K}_m(z_n^G) -  \mathcal{K}_m(\varphi_m) - {\mathcal{K}}_m'(\varphi_m) (z_n^G - \varphi_m) \right]	\\
				= \frac{1}{2} \mathcal{T} ( \varphi ) h^{ 2 r } + O \left(\max\left\{ h^{2 r + 2}, \tilde{h}^2 \right\}\right),
			\end{multline*}
			which completes the proof.
		\end{proof}

		\begin{lemma}\label{lem:4.4}
			If  $v \in \X$, then for $r \geq 1$,
			\begin{eqnarray}\notag
				\max_{0 \leq i \leq n} \left \lvert \mathcal{K}_m'(\varphi_m)  (I - P_n ) v (t_i) \right \rvert
				& \leq &  C_7 \norm{(I - P_n )v}_\infty  h^{r},
			\end{eqnarray}
			where $C_7$ is a constant independent of $h$.
		\end{lemma}
		\begin{proof} We write \begin{equation}\label{eq:4.19}
				\mathcal{K}_m'(\varphi_m)  (I - P_n ) v = \mathcal{K}_m'(\varphi)  (I - P_n ) v  + [\mathcal{K}_m'(\varphi_m) - \mathcal{K}_m'(\varphi)]  (I - P_n ) v.
			\end{equation}
			For fixed $ s \in [0, 1]$, let
			\begin{equation}\nonumber
				\ell_{*, s} (t) = \ell_* (s, t) = \ell (s, t, \varphi (t)) = \frac{\partial \kappa}{\partial u} (s, t, \varphi(t)), \quad t \in [0, 1].
			\end{equation} From the definition of $\mathcal{K}_m'(\varphi)$ and the discrete inner product, we have
			\begin{eqnarray}\nonumber
				\mathcal{K}_m'(\varphi)  (I - P_n ) v (s) 
				& = & \sum_{j=1}^n \inp {  \ell_{*, s}} {(I - P_{n,j}  )v}_{\Delta_j, m}. 
			\end{eqnarray} 
			Since $P_{n, j}$ is self-adjoint on $C(\Delta_j)$, so as $I - P_{n, j}$. Therefore
			\begin{eqnarray} \nonumber
				\mathcal{K}_m'(\varphi) (\I - P_n ) v (s) 
				& = & \sum_{j=1}^n \inp {(I - P_{n,j}) \ell_{*, s}} {(I - P_{n,j}) v}_{\Delta_j, m}.
			\end{eqnarray} 
			Note that, if  $ s = t_i$ for some $i \in \left\{ 0, 1, \ldots, n \right\}$, then 
			$\ell_{*, s} \in C^r[t_{j-1}, t_j]$ for all $j = 1, \ldots, n.$
			Hence from \eqref{proj:error1},
			\begin{eqnarray*}
				\|(I - P_{n,j}  ) \ell_{*, t_i}\|_{\Delta_j, \infty} & \leq & C_1 \left ( \sup_{t \in [t_{j-1}, t_j] } \lvert D^{(0, r)} \ell_* (t_i, t)  \rvert \right ) h^{r} 
			\end{eqnarray*}
			Thus,
			\begin{align*}
				\max_{0 \leq i \leq n} \left \lvert \mathcal{K}_m'(\varphi)  (I - P_n ) v (t_i) \right \rvert & \leq   \sum_{j=1}^n \norm{(I - P_{n,j}  ) \ell_{m, s}}_{\Delta_j, \infty} 
				\| (I - P_{n,j}) v\|_{\Delta_j, \infty} h \\
				& \leq   C_6 \norm{(I-P_n)v}_\infty   h^{ r },
			\end{align*}
			where $C_6$ is a constant independent of $h$. Now, from \eqref{eq:Ny3}, \eqref{eq:4.19} and the above estimate, we obtain
			$$ \max_{0 \leq i \leq n} \left \lvert \mathcal{K}_m'(\varphi_m)  (I - P_n ) v (t_i) \right \rvert
			\leq   C_7 \norm{(I - P_n )v}_\infty \max\left\{ h^{r}, \tilde{h}^2\right\},$$
			where $C_7 = C_4 + C_6$. Since $h^{r} \geq  \tilde{h}^2$, the result follows.
		\end{proof}
		Recall that $\mathcal{L}_m = \left[ I - {\mathcal{K}}_m'(\varphi_m) \right]^{-1} {\mathcal{K}}_m'(\varphi_m)$. Therefore, the proof of the following result is similar to that of the above lemma.
		\begin{Corollary}\label{cor:4.2}
			If  $v \in \X$, then for $r \geq 1$,
			\begin{eqnarray}\notag
				\max_{0 \leq i \leq n} \lvert \mathcal{L}_m  (I - P_n ) v (t_i) \rvert
				& \leq &  C_8 \norm{(I - P_n )v}_\infty  h^{r},
			\end{eqnarray}
			where $C_8$ is a constant independent of $h$.
		\end{Corollary}
		
		\begin{lemma}\label{lem:4.5} Let $v \in \X$. If $r = 1$, that is, when the range of $P_n$ is the space of piecewise polynomials of degree zero, then
			\begin{equation*}
				\norm{(I -P_n)\mathcal{K}_m''(\varphi) (v, v)}_\infty \leq C_9 h \norm{v}^2_\infty,
			\end{equation*}
			where $C_9$ is a constant independent of $h$.
		\end{lemma}
		\begin{proof} Given that $\mathcal{X}_n$ is the space of piecewise constant functions with respect to the partition \eqref{eq:partition1}. 
			Note that
			\begin{eqnarray*}
				\norm{(I -P_n)\mathcal{K}_m''(\varphi) (v, v)}_\infty
				& = & \max _{1 \leq j \leq n} \sup_{s \in [t_{j-1}, t_j]}
				\lvert (I -P_{n, j})\mathcal{K}_m''(\varphi) (v, v) (s) \rvert.
			\end{eqnarray*}
			Let $s \in [t_{j-1}, t_j]$. Since the Legendre polynomial of zero $L_0 (t) = 1$ for all $t \in [0, 1]$, we have from \eqref{eq:3.15}, 
			\begin{multline}\label{eqn:4.28}
				(I -P_n)\mathcal{K}_m''(\varphi) (v, v)(s) \\ = \frac{1}{p} \: \sum_{\nu =1}^p 	\sum_{q =1}^\rho w_q \left[\mathcal{K}_m''(\varphi) (v, v) (s)  - \mathcal{K}_m''(\varphi) (v, v)(t_{j-1} + \mu_{q \nu} h) \right].
			\end{multline}
			We also have
			{\small	\begin{multline*}
					\mathcal{K}_m''(\varphi) (v, v) (s)  - \mathcal{K}_m''(\varphi) (v, v)(t_{j-1} + \mu_{q \nu} h) \\ 
					=  \frac{h}{p} \sum_{k=1}^n  \sum_{q=1}^\rho \sum_{\nu =1}^p w_q \left[ \frac{\partial^2 \kappa}{\partial u^2} { \left( s,   \mu_{q \nu}^k, \varphi{\left( \mu_{q \nu}^k \right)}\right)} - \frac{\partial^2 \kappa}{\partial u^2} { \left( \mu_{q \nu}^j,   \mu_{q \nu}^k, \varphi{\left( \mu_{q \nu}^k \right)}\right)} \right] v^2{\left( \mu_{q \nu}^k \right)},
			\end{multline*}	}
			where $\mu_{q \nu}^j = t_{j-1} + \mu_{q \nu} h$.		
			For fixed $ s \in [0, 1]$, let
			\begin{equation}\nonumber
				\lambda_{*, s} (t) = \lambda_* (s, t) = \lambda (s, t, \varphi (t)) = \frac{\partial^2 \kappa}{\partial u^2} (s, t, \varphi(t)), \quad t \in [0, 1].
			\end{equation}
			Then, by \eqref{eq:dis_ip}
			\begin{align} \label{eqn:4.29}
				\mathcal{K}_m''(\varphi)(v, v)(s) & - \mathcal{K}_m''(\varphi) (v, v)(\mu_{q \nu}^j) 
				=  \sum_{k=1}^n  \left< \lambda_{*, s} - \lambda_{*, \mu_{q \nu}^i}, v^2\right>_{\Delta_k, m} 
				\nonumber \\
				&  = \sum_{ \stackrel {k =1} {k \neq j}}^n
				\inp {\lambda_{*,s} - \lambda_{*,\mu_{q \nu}^j } } {v^2}_{\Delta_k, m}
				+  \inp {\lambda_{*,s} - \lambda_{*, \mu_{q \nu}^j} } { v^2}_{\Delta_j, m}.
			\end{align}
			First consider the case when $k \neq j$. Applying Mean Value Theorem on the first component of $\lambda_{*}(\cdot, \cdot)$ in the interval $[s, \mu_{q \nu}^j]$, we obtain
			\begin{equation*}
				\inp {\lambda_{*,s} - \lambda_{*,\mu_{q \nu}^j} } {v^2}_{\Delta_k, m} = (s - \mu_{q \nu}^j) \inp {D^{(1, 0)}  \lambda_{*}  (\theta_{q \nu}^j, \cdot) } {v^2}_{\Delta_k, m},
			\end{equation*}
			for some $\theta_{q \nu}^j \in (t_{j - 1}, t_j)$, and the function $\displaystyle{D^{(1, 0)} \lambda_{*}(s,t)}, t \in [t_{k-1}, t_k]$ is given by
			\begin{equation*}
				D^{(1, 0)} \lambda_{*}(s,t) = \left\{ {\begin{array}{ll}	\vspace{1.5mm}
						D^{(1, 0)} \lambda_{1, *} (s, t) = \frac{\partial}{\partial s}\lambda_1 (s, t, \varphi (t)), \;\;\; 0 \leq t \leq s \leq 1, \\ 
						D^{(1, 0)} \lambda_{2, *} (s, t) = \frac{\partial}{\partial s}\lambda_2 (s, t, \varphi (t)), \;\;\; 0 \leq s \leq t \leq 1.
				\end{array}}\right.
			\end{equation*}
			Therefore, for $k \neq j$,
			\begin{eqnarray*}
				\left \lvert \inp {\lambda_{*,s} - \lambda_{*,\mu_{q \nu}^j} } {v^2}_{\Delta_k, m} \right \rvert &\leq & \left \lvert s - \mu_{q \nu}^j \right \rvert \left( \sup_{s \neq t} \left \lvert D^{(1, 0)} \lambda_{*} (s, t)
				\right \rvert \right) \norm{v}^2_\infty h \\
				& \leq & \left( \sup_{s \neq t} \left \lvert D^{(1, 0)} \lambda_{*} (s, t)
				\right \rvert \right) \norm{v}^2_\infty h^2,
			\end{eqnarray*}
			where  \begin{align*}
				&  \sup_{s \neq t} \lvert D^{(1, 0)} \lambda_{*} (s, t)
				\rvert \\
				& =\max\left\{ \sup_{0\leq t<s \leq 1} \lvert D^{(1,0)}\lambda_{1,*}{(s,t)}\rvert~,\sup_{0\leq s < t \leq 1} \lvert D^{(1,0)}\lambda_{2,*}{(s,t)}\rvert \right\} \\
				& =\max\left\{ \sup_{\stackrel{0\leq t < s \leq 1}{ |u| \leq \norm{\varphi}_\infty}} \lvert  D^{(1,0, 2)}\kappa_{1}{(s,t,u)} \rvert ~, \sup_{\stackrel{0\leq s < t \leq 1}{ |u| \leq \norm{\varphi}_\infty}} \lvert D^{(1,0,2)}\kappa_{2}{(s,t,u)}\rvert \right\}.
			\end{align*}
			On the other hand,
			\begin{align*}
				\lvert \inp {\lambda_{*,s} - \lambda_{*, \mu_{q \nu}^j} } { v^2}_{\Delta_j, m} \rvert \leq 2 \left( \sup_{0 \leq s, t \leq 1} \lvert \lambda_{*} (s, t)
				\rvert \right) \norm{v}^2_\infty h,
			\end{align*}
			where $ \displaystyle{\sup_{0 \leq s, t \leq 1} \lvert \lambda_{*} (s, t)
				\rvert = \sup_{\stackrel{0\leq s , t \leq 1}{ |u| \leq \norm{\varphi}_\infty}} \lvert D^{(0,0,2)}\kappa{(s,t,u)}\rvert}$.
			Then, from \eqref{eqn:4.29} we obtain
			\begin{multline*}
				\lvert \mathcal{K}_m''(\varphi)(v, v)(s)  - \mathcal{K}_m''(\varphi) (v, v)(\mu_{q \nu}^j)\rvert \\ \leq \left( \sum_{ \stackrel {k =1} {k \neq j}}^n
				\left( \sup_{s \neq t} \lvert D^{(1, 0)} \lambda_{*} (s, t) \vert \right) \norm{v}^2_\infty h^2 \right)
				+   2 \left( \sup_{0 \leq s, t \leq 1} \lvert \lambda_{*} (s, t)
				\rvert \right) \norm{v}^2_\infty h.
			\end{multline*}
			It follows that
			\begin{eqnarray*}
				\lvert \mathcal{K}_m''(\varphi)(v, v)(s)  - \mathcal{K}_m''(\varphi) (v, v)(\mu_{q \nu}^j)\rvert  \leq C_9 \norm{v}^2 h,
			\end{eqnarray*}
			where $\displaystyle{C_9 = \left( \sup_{s \neq t} \vert D^{(1, 0)} \lambda_{*} (s, t)\rvert \right) + 2 \left( \sup_{0 \leq s, t \leq 1} \lvert \lambda_{*} (s, t)
				\rvert \right)}$.\\
			The result now follows from \eqref{eqn:4.28} and the above estimate.
		\end{proof}
		
		\begin{Corollary}\label{cor:4.5}
			Let $v \in \X$. If $r = 1$, that is, when the range of $P_n$ is the space of piecewise polynomials of degree zero, then
			\begin{equation*}
				\norm{(I -P_n)\mathcal{K}_m'(\varphi) v}_\infty \leq C_{10} h \norm{v}_\infty,
			\end{equation*}
			where $C_{10}$ is a constant independent of $h$.
		\end{Corollary}
		\begin{proof}
			The proof is similar to that of Lemma \ref{lem:4.5}.
		\end{proof}

		\begin{proposition}\label{prop:5}Let $t_i$ be any point of the partition $\Delta^{(n)}$ defined by \eqref{eq:partition1}. Then
			\begin{multline*}
				\mathcal{L}_m (I -P_n) \left[\mathcal{K}_m(z_n^G) -  \mathcal{K}_m(\varphi_m) - {\mathcal{K}}_m'(\varphi_m) (z_n^G - \varphi_m) \right](t_i) \\ = \left\{ {\begin{array}{ll}	\vspace{1mm}
						O\left( h^4 \right), \hspace*{2.65cm}r = 1, \\ 
						O\left(\max\left\{ h^{3r}, h^r\tilde{h}^4 \right\}  \right), \quad r \geq 2.
				\end{array}}\right.
			\end{multline*}
		\end{proposition}
		\begin{proof}
			Generalized Taylor's series expansion gives
			\begin{multline}\label{eqn:4.30}
				\mathcal{L}_m (I -P_n) \left[\mathcal{K}_m(z_n^G) -  \mathcal{K}_m(\varphi_m) - \mathcal{K}_m'(\varphi_m) (z_n^G - \varphi_m) \right] \\ = \frac{1}{2}\mathcal{L}_m (I -P_n)\mathcal{K}_m''(\varphi_m) (z_n^G - \varphi_m)^2 + \mathcal{L}_m (I -P_n)\mathcal{R}_{3, m} \left( z_n^G - \varphi_m \right),
			\end{multline}	
			where
			$$\mathcal{R}_{3, m} \left( z_n^G - \varphi_m \right) =  \int_0^1 \frac{(1 - \theta)^2}{2!}\mathcal{K}_m^{(3)} \left(\varphi_m + \theta (z_n^G - \varphi_m) \right) (z_n^G - \varphi_m)^3
			\; d \theta.$$		
			It follows that
			\begin{eqnarray*}
				\norm{\mathcal{R}_{3, m} \left( z_n^G - \varphi_m \right)}_\infty \leq \frac{1}{6} \left( \sup_{\stackrel {s, t \in [0, 1]}{\lvert u \rvert \leq \|\varphi \|_\infty + \epsilon}} 	\left\lvert  \frac{\partial^3 \kappa}{\partial u^3}(s, t, u) \right\rvert \right) \norm{z_n^G - \varphi_m}^3_\infty.
			\end{eqnarray*}
			Therefore, by \eqref{eq:discrete_Gal}
			\begin{equation*}
				\norm{\mathcal{R}_{3, m} \left( z_n^G - \varphi_m \right)}_\infty = O\left( \max\left\{ h^{3r}, \tilde{h}^6 \right\}  \right).
			\end{equation*}
			Since $\norm{I - P_n} \leq 1 + \norm{P_n} < \infty$, from Corollary \ref{cor:4.2}, it is easy to see that
			\begin{align}\label{eqn:4.31}
				\mathcal{L}_m (I -P_n)\mathcal{R}_{3, m} \left( z_n^G - \varphi_m \right)(t_i) = O\left(\max\left\{ h^{4r}, h^r\tilde{h}^6 \right\}  \right).
			\end{align}
			First consider the case $r\geq 2$. Since $\norm{\mathcal{K}_m''(\varphi_m)} < \infty$ and $\norm{I - P_n}_\infty < \infty$, by  \eqref{eq:discrete_Gal} and the Corollary \ref{cor:4.2}, we have
			\begin{align}\nonumber
				\frac{1}{2} \mathcal{L}_m (I -P_n)\mathcal{K}_m''(\varphi_m) (z_n^G - \varphi_m)^2 (t_i) = O\left(\max\left\{ h^{3r}, h^r\tilde{h}^4 \right\}  \right), \quad r \geq 2.
			\end{align}
			When $r = 1$, we write
			{\small \begin{multline*}
					\mathcal{L}_m (I -P_n)\mathcal{K}_m''(\varphi_m) (z_n^G - \varphi_m)^2  \\ = \mathcal{L}_m (I -P_n)\left[ \mathcal{K}_m''(\varphi_m) - \mathcal{K}_m''(\varphi) \right] (z_n^G - \varphi_m)^2  + \mathcal{L}_m (I -P_n)\mathcal{K}_m''(\varphi) (z_n^G - \varphi_m)^2 .
			\end{multline*}}
			By \eqref{eq:discrete_Gal}, \eqref{Gr_Ny_error} and the Lemma \ref{lip_2}, we have
			\begin{equation*}
				\mathcal{L}_m (I -P_n)\left[ \mathcal{K}_m''(\varphi_m) - \mathcal{K}_m''(\varphi) \right] (z_n^G - \varphi_m)^2 = O\left( h^4 \right).
			\end{equation*}
			On the other hand 
			{\small	\begin{equation*}
					\mathcal{L}_m (I -P_n)\mathcal{K}_m''(\varphi) (z_n^G - \varphi_m)^2 (t_i)  = \sum_{j= 1}^n \inp { \ell_{m, t_i}} {(I - P_{n,j})\mathcal{K}_m''(\varphi) (z_n^G - \varphi_m)^2}_{\Delta_j, m}.
			\end{equation*}}
			Since $I - P_{n,j}$ is self-adjoint,
			\begin{multline*}
				\mathcal{L}_m (I -P_n)\mathcal{K}_m''(\varphi) (z_n^G - \varphi_m)^2 (t_i) \\ = \sum_{j= 1}^n \inp { (I - P_{n,j})\ell_{m, t_i}} {(I - P_{n,j})\mathcal{K}_m''(\varphi) (z_n^G - \varphi_m)^2}_{\Delta_j, m}.
			\end{multline*}
			It follows that
			\begin{multline*}
				\max_{0 \leq i \leq n} \left \lvert \mathcal{L}_m (I -P_n)\mathcal{K}_m''(\varphi) (z_n^G - \varphi_m)^2 (t_i) \right \rvert  \\ \leq \sum_{j= 1}^n \norm{(I - P_{n,j})\ell_{m, t_i}}_{\Delta_j, \infty} \norm{(I - P_{n,j})\mathcal{K}_m''(\varphi) (z_n^G - \varphi_m)^2}_{\Delta_j, \infty} h.
			\end{multline*}
			By Corollary \ref{cor:4.2} and Lemma \ref{lem:4.5}, we obtain
			\begin{align*}
				\max_{0 \leq i \leq n} \lvert \mathcal{L}_m (I -P_n)\mathcal{K}_m''(\varphi) (z_n^G - \varphi_m)^2 (t_i) \rvert = O\left( h^4 \right), \quad r = 1.
			\end{align*}
			Therefore
			\begin{align}\label{eqn:4.32}
				\frac{1}{2} \mathcal{L}_m (I -P_n)\mathcal{K}_m''(\varphi_m) (z_n^G - \varphi_m)^2 (t_i) = \left\{ {\begin{array}{ll}	\vspace{1mm}
						O\left( h^4 \right), \hspace*{2.62cm}r = 1, \\ 
						O\left(\max\left\{ h^{3r}, h^r\tilde{h}^4 \right\}  \right), \quad r \geq 2.
				\end{array}}\right.
			\end{align}
			Then combing \eqref{eqn:4.30}, \eqref{eqn:4.31} and \eqref{eqn:4.32}, we obtain
			\begin{multline*}
				\mathcal{L}_m (I -P_n) \left[\mathcal{K}_m(z_n^G) -  \mathcal{K}_m(\varphi_m) - {\mathcal{K}}_m'(\varphi_m) (z_n^G - \varphi_m) \right](t_i) \\ = \left\{ {\begin{array}{ll}	\vspace{1mm}
						O\left( h^4 \right), \hspace*{2.62cm}r = 1, \\ 
						O\left(\max\left\{ h^{3r}, h^r\tilde{h}^4 \right\}  \right), \quad r \geq 2.
				\end{array}}\right.
			\end{multline*}
			This follows the result.
		\end{proof}

		We quote the following result from By \cite[Proposition 1, Proposition 6]{RPK-GR3}, which will be used in the next proposition.
		\begin{equation}\label{eqn:4.33}
			\norm{(I -P_n)\mathcal{K}_m'(\varphi)(I -P_n)\varphi}_\infty  = \left\{ {\begin{array}{ll}	\vspace{1mm}
					O\left( h^3 \right), \hspace*{7mm}r = 1, \\ 
					O\left( h^{r+2} \right), \quad r \geq 2.
			\end{array}}\right.
		\end{equation}

		\begin{proposition}\label{prop:6} If $\varphi_m$ and $z_n^G$ are respectively the Nystr\"om and the discrete Galerkin approximation of $\varphi$, then
			\begin{equation*}
				\mathcal{L}_m (I -P_n){\mathcal{K}}_m'(\varphi_m) (z_n^G - \varphi_m)(t_i) = O \left(\max\left\{ h^{2r+2}, \tilde{h}^2 \right\}\right).
			\end{equation*}
		\end{proposition}
		\begin{proof}
			Adding and subtracting $\mathcal{K}_m'(\varphi)$, we have
			\begin{multline*}
				\mathcal{L}_m (I -P_n){\mathcal{K}}_m'(\varphi_m) (z_n^G - \varphi_m) \\= \mathcal{L}_m (I -P_n)\left[\mathcal{K}_m'(\varphi_m) - \mathcal{K}_m'(\varphi)\right] (z_n^G - \varphi_m) + \mathcal{L}_m (I -P_n)\mathcal{K}_m'(\varphi) (z_n^G - \varphi_m). 
			\end{multline*}
			Then, using \eqref{eq:discrete_Gal}, \eqref{eq:Ny3} and the Corollary \ref{cor:4.2}, we obtain for $r \geq 1$,
			\begin{multline}\label{eqn:4.34}
				\left \lvert \mathcal{L}_m (I -P_n)\left[\mathcal{K}_m'(\varphi_m) - \mathcal{K}_m'(\varphi)\right] (z_n^G - \varphi_m)(t_i) \right \rvert \\ \leq C_4 C_8 \left( 1 + \norm{P_n}\right) h^r \tilde{h}^2 \left(\max\left\{ h^{r}, \tilde{h}^2 \right\}  \right)
			\end{multline}
			Note that 
			\begin{align*}
				z_n^G - \varphi_m  = P_n z_n^S - \varphi_m
				& = P_n z_n^S - P_n \varphi + P_n \varphi - \varphi + \varphi - \varphi_m \\
				& = P_n \left( z_n^S - \varphi \right) - \left( I - P_n \right)\varphi + \left( \varphi - \varphi_m \right).
			\end{align*}
			Then
			\begin{align*}
				\mathcal{L}_m (I -P_n)\mathcal{K}_m'(\varphi) (z_n^G - \varphi_m) & = \mathcal{L}_m (I -P_n)\mathcal{K}_m'(\varphi)P_n \left( z_n^S - \varphi \right) \\ & - \mathcal{L}_m (I -P_n)\mathcal{K}_m'(\varphi)\left( I - P_n \right)\varphi \\ &  + \mathcal{L}_m (I -P_n)\mathcal{K}_m'(\varphi)\left( \varphi - \varphi_m \right).
			\end{align*}
			By the Corollary \ref{cor:4.2}
			\begin{align*}
				\left \lvert \mathcal{L}_m (I -P_n)\mathcal{K}_m'(\varphi)P_n \left( z_n^S - \varphi \right)(t_i)  \right \rvert \leq C_8 \norm{\left( I - P_n \right)\mathcal{K}_m'(\varphi)P_n \left( z_n^S - \varphi \right)} h^r,
			\end{align*}
			then by \eqref{eq:discrete_it_Gal} and Corollary \ref{cor:4.5}, we obtain
			\begin{align*}
				\left \lvert \mathcal{L}_m (I -P_n)\mathcal{K}_m'(\varphi)P_n \left( z_n^S - \varphi \right)(t_i)  \right \rvert  & = O \left(\max\left\{ h^{2r+2}, \tilde{h}^2 \right\}\right), \quad \text{ for } r\geq 1.
			\end{align*}
			Also, the Corollary \ref{cor:4.2} and \eqref{eqn:4.33} implies
			\begin{align*}
				\mathcal{L}_m (I -P_n)\mathcal{K}_m'(\varphi)\left( I - P_n \right)\varphi (t_i) = 	O\left( h^{2r+2} \right), \quad \text{ for } r \geq 1.
			\end{align*}
			It is easy to see (from \eqref{Gr_Ny_error} and Corollary \ref{cor:4.2}) that
			\begin{equation} \nonumber
				\mathcal{L}_m (I -P_n)\mathcal{K}_m'(\varphi)\left( \varphi - \varphi_m \right) = 	O\left( h^{r} \tilde{h}^2 \right), \quad \text{ for } r \geq 1.
			\end{equation}
			Therefore, for $r \geq 1$,
			\begin{align}\nonumber
				\mathcal{L}_m (I -P_n)\mathcal{K}_m'(\varphi) (z_n^G - \varphi_m) (t_i) = O \left(\max\left\{ h^{2r+2}, \tilde{h}^2 \right\}\right).
			\end{align}
			Hence, the result follows from \eqref{eqn:4.34} and the above equation.
		\end{proof}

		We prove the main theorem as follows.
		
		\begin{theorem}\label{thm:1}
			Let $\mathcal{K}$ be the Urysohn integral operator with Green's function type kernel $\kappa$, defined by \eqref{eq:1.1}.   
			Let $\varphi$ be the unique solution of the equation (\ref{eq:main}). Assume that $1$ is not an eigenvalue of $\mathcal{K}' (\varphi).$  Let $\mathcal{X}_n$ be the space of piecewise polynomials of degree $\leq r -1 $ with respect to the partition $\Delta^{(n)} : = 0 = t_0 < t_1 < \cdots < t_n = 1 $ defined by \eqref{eq:partition1}. Let $P_n: L^\infty [0, 1] \rightarrow \mathcal{X}_n$ be the discrete orthogonal projection defined by \eqref{dop} and $\tilde{z}_n^S$ be the discrete iterated Galerkin approximation of $\varphi$. Then
			\begin{eqnarray}\nonumber
				\left( z_n^S - \varphi \right)(t_i) = \left[  \mathcal{E}_{2r}(\varphi)(t_i) + \frac{1}{2} \mathcal{T}(\varphi)(t_i) \right] h^{2r} + O \left(\max\left\{ h^{2r+2}, \tilde{h}^2 \right\}\right),
			\end{eqnarray}
			where the operators $\mathcal{E}_{2r}$ and $\mathcal{T}$ are respectively defined by \eqref{asy_exp1} and \eqref{eq:exp2}.
		\end{theorem}
		\begin{proof}
			We have from \eqref{equation:1}
			\begin{align*}
				z_n^S - \varphi = & \left[ I - {\mathcal{K}}_m'(\varphi_m) \right]^{-1} \left[\mathcal{K}_m(z_n^G) -  \mathcal{K}_m(\varphi_m) - {\mathcal{K}}_m'(\varphi_m) (z_n^G - \varphi_m) \right] \\
				& - \mathcal{L}_m (I -P_n) \left[\mathcal{K}_m(z_n^G) -  \mathcal{K}_m(\varphi_m) - {\mathcal{K}}_m'(\varphi_m) (z_n^G - \varphi_m) \right] \\
				& - \mathcal{L}_m (I -P_n){\mathcal{K}}_m'(\varphi_m) (z_n^G - \varphi_m) \\
				& - \mathcal{L}_m (I -P_n) \varphi_m \\
				& + \varphi_m - \varphi.
			\end{align*}
			The result now follows from \eqref{Gr_Ny_error}, Proposition \ref{prop:3}, Proposition \ref{prop:4}, Proposition \ref{prop:5} and Proposition \ref{prop:6}.
		\end{proof}
		
		We now apply Richardson extrapolation to obtain an approximation of $\varphi$ with higher order of convergence. Define
		$$ z_n^{EX} = \frac{2^{4r}z_{2n}^S - z_n^S}{2^{4r} - 1}.$$
		We choose the partitions $\Delta^{(m)}$ and $\Delta^{(n)}$ such that $m^2 \geq n^{2r+2}$. Then, it is easy to see from the Theorem \ref{thm:1}, that
		\begin{equation}\label{Ex_order}
			\left(z_n^{EX} - \varphi\right)(t_i) =  O\left(h^{2r+2}\right), \quad \text{ for all } i =1, 2, \dots, n.
		\end{equation}

		\section{Numerical results}
		
		For the numerical results, we consider the following example from \cite{Rpk-Aks}. 
		\noindent
		Consider
		\begin{equation}\label{eq:4.1}
			\varphi (s) - \int_0^1 \kappa (s, t) \left[ \psi  \left(t, \varphi (t) \right) \right] \: dt  = f(s), \;\;\; 0 \leq s \leq 1,
		\end{equation}
		where 
		$$\kappa (s,t) =\frac{1}{\gamma \sinh \gamma} \left\{ {\begin{array}{ll}
				\sinh \gamma s \: \sinh \gamma(1-t), & ~ 0 \leq t \leq s \leq 1, \\
				\gamma (1-s) \sinh \gamma t, & ~ 0 \leq s \leq t \leq 1,
		\end{array}}\right. $$
		with $\gamma = \sqrt{12},$
		and $$  \psi(t, \varphi(t))= \gamma^2 \varphi(t) - 2 \left( \varphi(t) \right)^3, \quad t \in [0,1].$$
		We have $f(s) =\frac{1}{ \sinh \gamma} \left \lbrace 2 \sinh \gamma(1-s) + \frac{2}{3} \sinh \gamma s \right \rbrace. $ The exact solution of \eqref{eq:4.1} is given by $$\varphi(s) =\frac{2}{2s+1}, \quad s \in [0, 1].$$

		Let $\mathcal{X}_n$ be the space of piecewise constant functions with respect to the uniform partition $\Delta^{(n)}$ of the interval $[0,1]$. Let $P_n: L^\infty[0,1] \rightarrow \mathcal{X}_n$ be the discrete orthogonal projection defined by \eqref{dop}. 
		
		Let $t_i=\frac{i-1}{20}, i=1,2,\ldots,21$ be the partition points with step size $h=\frac{1}{20}.$ The numerical quadrature is chosen to be the composite 2 point Gaussian quadrature rule with respect to partition $\Delta^{(m)}$ with $m = n^2$ subintervals. Then $\tilde{h} = h^2$. Therefore, it is expected from the Theorem \ref{thm:1} and equation \eqref{Ex_order}, that
		\begin{equation*}
			\epsilon_n^S(t_i) =\lvert \varphi(t_i) - z_n^S(t_i)\rvert = O\left( h^2 \right) ~ \text{ and } ~ \epsilon_n^{EX}(t_i) = \lvert \varphi(t_i) - z_n^{EX}(t_i) \rvert = O\left( h^4 \right),
		\end{equation*}
		where  $$z_n^{EX}(t_i) = \frac{4 z_{2n}^S(t_i) -z_n^S(t_i)}{3}.$$
		%The orders of convergence are calculated using the formula :
		%\begin{align*}
		%	\begin{aligned}
			%		\delta^S_1= \frac{log\left(\epsilon_n^S(t_i)/\epsilon_{2n}^S(t_i)\right)}{log(2)}, \\ 
			%		\delta^{EX}=\frac{log\left(\epsilon_n^{EX}(t_i)/\epsilon_{2n}^{EX}(t_i)\right)}{log(2)},  
			%	\end{aligned} ~~~ ~~~~ n=40	
		%\end{align*}
		%\begin{align*}
		%	\delta^S_2= \frac{log\left(\epsilon_{n}^S(t_i)/\epsilon_{2n}^S(t_i)\right)}{log(2)}, \quad n=20.
		%\end{align*} 
		%We expect $\delta^S_1 = \delta^S_2 = 2$ and $	\delta^{EX} = 4.$
		Let $\delta^S$ and $\delta^{EX}$ be respectively the orders of convergence of $z_n^S$ and $z_n^{EX}$ at the partition points. We expect $\delta^S = 2$ and $\delta^{EX} = 4.$
		
		%	\vspace*{6mm}	
		\begin{center}
			
			Table 1
			
			\begin{tabular} {|c|c|c|c|c|c|}\hline
				$t_i$ & $\epsilon_n^S(t_i): n=20$ & ~~~ $\delta^S$ & $\epsilon_n^{EX}(t_i): n=20$ & ~~~ $\delta^{EX}$
				\\
				\hline
				$0.05$&  $  8.6 \times 10^{-3}	 $ &  $2.00$ &  $  2.98 \times 10^{-6}   $& $3.99$               \\
				$0.1$&  $  7.56 \times 10^{-3}	 $ &  $2.00$ &  $  2.23 \times 10^{-6}	 $& $3.99$               \\
				$0.15$&  $  6.79 \times 10^{-3}	 $ &  $2.00$ &  $  1.59 \times 10^{-6}	 $& $3.99$               \\
				$0.2$&  $  6.22 \times 10^{-3}	 $ &  $2.00$ &  $  1.09 \times 10^{-6}	 $& $3.97$               \\
				$0.25$&  $  5.78 \times 10^{-3}	 $ &  $2.00$ &  $  7.13 \times 10^{-7}	 $& $3.96$               \\
				$0.3$&  $  5.45 \times 10^{-3}	 $ &  $2.00$ &  $  4.46 \times 10^{-7}	 $& $3.94$               \\
				$0.35$&  $  5.19 \times 10^{-3}	 $ &  $2.00$ &  $  2.7 \times 10^{-7}	 $& $3.91$               \\
				$0.4$&  $  4.98 \times 10^{-3}	 $ &  $2.00$ &  $  1.69 \times 10^{-7}	 $& $3.86$               \\
				$0.45$&  $  4.82 \times 10^{-3}	 $ &  $2.00$ &  $  1.3 \times 10^{-7}	 $& $3.83$               \\
				$0.5$&  $  4.68 \times 10^{-3}	 $ &  $2.00$ &  $  1.41 \times 10^{-7}	 $& $3.85$               \\
				$0.55$&  $  4.55 \times 10^{-3}	 $ &  $2.00$ &  $  1.91 \times 10^{-7}	 $& $3.89$               \\
				$0.6$&  $  4.44 \times 10^{-3}	 $ &  $2.00$ &  $  2.72 \times 10^{-7}	 $& $3.93$               \\
				$0.65$&  $  4.33 \times 10^{-3}	 $ &  $2.00$ &  $  3.75 \times 10^{-7}	 $& $3.95$               \\
				$0.7$&  $  4.22 \times 10^{-3}	 $ &  $2.00$ &  $  4.95 \times 10^{-7}	 $& $3.97$               \\
				$0.75$&  $  4.10 \times 10^{-3}	 $ &  $2.00$ &  $  6.26 \times 10^{-7}	 $& $3.98$               \\
				$0.8$&  $  3.98 \times 10^{-3}	 $ &  $2.00$ &  $  7.6 \times 10^{-7}	 $& $3.99$               \\
				$0.85$&  $  3.84 \times 10^{-3}	 $ &  $2.00$ &  $  8.94 \times 10^{-7}	 $& $3.99$               \\
				$0.9$&  $  3.69 \times 10^{-3}	 $ &  $2.00$ &  $  1.02 \times 10^{-6}	 $& $3.99$               \\
				$0.95$&  $  3.52 \times 10^{-3}	 $ &  $2.00$ &  $  1.14 \times 10^{-6}	 $& $4$               \\
				\hline
				
			\end{tabular}
		\end{center}
		
		From the above table, it is clear that the obtained orders of convergence match well with the theoretical orders of convergence. Also the order of convergence of the extrapolated solution improves upon the discrete iterated Galerkin solution.

		%\subsection*{Acknowledgment}

		% ------------------------------------------------------------------------
		
%		
%		\bibliographystyle{plain}
%		\nocite*{}
%		\bibliography{bibliography} 
		% ------------------------------------------------------------------------
	\end{document}